\documentclass[a4paper,10pt]{article}

\usepackage[margin=3cm,footskip=1cm]{geometry}

\usepackage{amsmath,amssymb,amsthm,bm,mathtools,xspace,booktabs,url}
\usepackage{graphicx,etoolbox}
\usepackage{hyperref}

\usepackage[shortlabels]{enumitem}

\setlist{
  listparindent=\parindent,
  parsep=0pt,
}
\usepackage{amssymb}

\usepackage[sort&compress]{natbib}		
\usepackage{cleveref}

\AtBeginEnvironment{thebibliography}{
  \small
  \setlength\itemsep{0.75em plus 0.5em minus 0.75em}%
}

\usepackage{caption}
\usepackage[raggedright]{titlesec}

\numberwithin{equation}{section}
\usepackage{amsfonts}

\theoremstyle{plain} 
\newtheorem{theorem}{Theorem}[section]
\newtheorem{Lemma}[theorem]{Lemma}
\newtheorem{Proposition}[theorem]{Proposition}
\newtheorem{Corollary}[theorem]{Corollary}
\newtheorem{definition}[theorem]{Definition}

\theoremstyle{definition} 
\newtheorem{Example}[theorem]{Example}
\newtheorem{Remark}[theorem]{Remark}

\usepackage{authblk}

\makeatletter
\newcommand\CorrespondingAuthor[1]{
  \begingroup
  \def\@makefnmark{}
  \footnotetext{Corresponding author: #1}
  \endgroup
}
\makeatother

\makeatletter
\renewenvironment{abstract}{%
  \small%
  \providecommand\keywords{%
    \par\medskip\noindent\textit{Keywords:}\xspace}%
  \begin{center}%
    \bfseries \abstractname\vspace{-.5em}\vspace{\z@}%
  \end{center}%
  \quote%
}
{
\endquote}
\makeatother

\usepackage[above,below,section]{placeins}

\usepackage[all]{onlyamsmath}

\usepackage{multirow}

\usepackage{xcolor}
\usepackage{color}

\usepackage{mathrsfs}
\DeclareMathAlphabet{\mathpzc}{T1}{pzc}{m}{it}

\newcommand{\RL}{\mathbb{R}}
\newcommand{\nat}{{\mathbb N}}

\newcommand{\e}{\mathrm{e}}
\newcommand{\dd}{\mathrm{d}}

\newcommand{\Exp}{\mathbb{E}}

\newcommand{\Prob}{\mathbb{P}}
\newcommand{\Ind}{\mathbb I}
\newcommand{\Ftail}{\overline F\,}

\def\*{\discretionary{}{\hbox{\ensuremath\cdot}\thinspace}{}}

\usepackage{tabularx,ragged2e,booktabs}
\newcommand\Tstrut{\rule{0pt}{2.6ex}}         
\newcommand\Bstrut{\rule[-0.9ex]{0pt}{0pt}}   

\usepackage{pdfpages}

\begin{document}

\title{Approximation of Ruin Probabilities via Erlangized Scale Mixtures}
\author[1]{\small Oscar Peralta}
\author[2]{\small Leonardo Rojas-Nandayapa}
\author[3]{\small Wangyue Xie}
\author[3]{\small Hui Yao}
\affil[1]{Department of Applied Mathematics and Computer Science, Technical University of Denmark, Denmark, osgu@dtu.dk}
\affil[2]{Mathematical Sciences, University of Liverpool, UK, leorojas@liverpool.ac.uk}
\affil[3]{School of Mathematics and Physics, The University of Queensland, 
   Australia, \authorcr w.xie1@uq.edu.au, h.yao@uq.edu.au}
\date{}
\maketitle
\begin{abstract}

In this paper, we extend an existing scheme for  
numerically calculating the probability of ruin of a classical Cram\'er--Lundberg reserve process
having absolutely continuous but otherwise general claim size distributions.
We employ a dense class of distributions that we denominate
\emph{Erlangized scale mixtures} (ESM) and correspond to 
nonnegative and absolutely continuous distributions which can be written as a Mellin--Stieltjes convolution $\Pi\star G$ of a 
nonnegative distribution $\Pi$ with an Erlang distribution $G$.
A distinctive feature of such a class is that it contains heavy-tailed distributions.

We suggest a simple methodology for constructing a sequence of distributions having the form $\Pi\star G$ to approximate the integrated tail distribution of the claim sizes.
Then we adapt a recent result which delivers an explicit expression for the probability of
ruin in the case that the  claim size distribution is
modelled as an Erlangized scale mixture.  We provide 
simplified expressions for the approximation of the probability of ruin and construct explicit bounds for the error of approximation.  We complement our results
with a classical example where the claim sizes are heavy-tailed.


\end{abstract}

\section{Introduction}

In this paper we propose a new numerical scheme for the approximation 
of ruin probabilities in the classical compound Poisson risk model --- also known as 
Cram\'er--Lundberg risk model \citep[cf.][]{Asmussen2010}.
In such a risk model, the surplus process is modelled as a compound Poisson process
with negative linear drift and a nonnegative jump distribution $F$, the later
corresponding to the claim size distribution. The ruin probability  
within infinite horizon and initial capital $u$, denoted $\psi(u)$,  is the probability 
that the supremum of the surplus process is larger than $u$. The Pollaczek--Khinchine formula provides the exact value of $\psi(u)$, though it can be explicitly computed in very few cases. Such a formula is a functional of ${\widehat F}$, the integrated tail distribution of $F$. From here on, we will use $\psi_{\widehat F}(u)$ instead of $\psi(u)$ to denote this dependence. A useful fact is that the Pollaczek--Khinchine formula can be naturally extended in order to define $\psi_{G}(u)$ even if $G$ does not correspond to an integrated tail distribution. We are doing so throughout this manuscript.

The approach advocated in this paper is to approximate the integrated claim size distribution
$\widehat F$ by using the family  of \emph{phase-type scale mixture distributions \cite{Bladt2015}}, 
but we also consider the more common approach of approximating the claim size distribution $F$. 
The family of phase-type scale mixture distributions is dense within the class of nonnegative distributions, 
and it is formed by distributions which can be expressed as a
Mellin--Stieltjes convolution, denoted $\Pi\star G$,  
of an arbitrary nonnegative distribution $\Pi$ and a phase-type distribution $G$
\citep[cf.][]{Bingham1987}.  The Mellin--Stieltjes convolution 
corresponds to the distribution of the product between two independent random variables 
having distributions $\Pi$ and $G$ respectively:

\begin{equation*}
 \Pi\star G(u):=\int_0^\infty G(u/s)d\Pi(s)=\int_0^\infty \Pi(u/s)dG(s).
\end{equation*}

In particular, if $\Pi$ is a nonnegative discrete distribution and $\Pi\star G$ is itself  the integrated tail of a phase-type scale mixture distribution, then
an explicit computable formula for the ruin probability $\psi_{\Pi\star G}(u)$ 
of the Cram\'er--Lundberg process having integrated tail distribution $\Pi\star G$ is
given in \cite{Bladt2015}.  Hence, it is plausible that if $\Pi\star G$ 
is \emph{close enough} to the integrated tail distribution $\widehat F$ of the 
claim sizes, then we can use $\psi_{\Pi\star G}(u)$ as an approximation for $\psi_{\widehat F}(u)$,
the ruin probability of a Cram\'er--Lundberg process having claim size distribution $F$.
One of the key features of the class of phase-type scale mixtures is that if $\Pi$ has unbounded
support, then $\Pi\star G$ is a  heavy-tailed distribution \citep{SuChen2006,Tang2008b,Nardo2015}, thus 
confirming the hypothesis that the class
of phase-type scale mixtures is more appropriate for approximating tail-dependent
quantities involving heavy-tailed distributions.  In contrast, the class of classical
phase-type distributions is light-tailed and approximations derived from this approach
may be inaccurate in the tails \citep[see also][for an extended discussion]{VatamidouAdanVlasiouZwart2014}.

Our contribution is to propose a systematic methodology to approximate any continuous integrated tail
distribution  $\widehat F$ using a  particular  subclass of phase-type scale mixtures called 
\emph{Erlangized scale mixtures} (ESM).  The proposed approximation is particularly precise
in the tails and the number of parameters remains controlled.
Our construction requires a sequence  $\{\Pi_m:m\in\nat\}$ of nonnegative discrete distributions 
having the property $\Pi_m\rightarrow \widehat F$ (often taken as a discretization of 
the target distribution over some countable subset of the support of $\widehat{F}$), and a sequence of Erlang distributions with equal shape and rate parameters, denoted
$G_m\sim \mbox{Erlang}(\xi(m),\xi(m))$. If the sequence $\xi(m)\in\nat$ is increasing and unbounded, then 
$\Pi_m\star G_m\rightarrow \widehat F$.  Then we can adapt the results in 
\cite{Bladt2015} to compute $\psi_{\Pi_m\star G_m}(u)$, and use this as an approximation of 
the ruin probability of interest.

To assess the quality of $\psi_{\Pi_m\star G_m}(u)$ as an approximation of $\psi_{\widehat F}(u)$ 
we identify two sources of theoretical error. The first source of error comes from approximating 
$\widehat F$ via $\Pi_m$, so we refer to this as
 the \emph{discretization error}. The second source of error is due to the convolution with $G_m$, so
this will be called the \emph{Erlangization error}.  The two errors are closely intertwined
so  it is difficult to make a precise assessment of the effect of each of them in
the general approximation.
Instead, we use the triangle inequality to separate these as follows
\begin{equation*}
\underbrace{\left|\psi_{\widehat F}(u)-\psi_{\Pi_m\star G_m}(u)\right|}_{\text{Approximation error}}
 \le \underbrace{\left|\psi_{\widehat F}(u)-\psi_{\widehat F\star G_m}(u)\right|}_{\text{Erlangization error}}
  +\underbrace{\left|\psi_{\widehat F\star G_m}(u)-\psi_{\Pi_m\star G_m}(u)\right|}_{\text{Discretization error}}.
\end{equation*}
Therefore, the error of approximating $\psi_{\widehat{F}}(u)$ with $\psi_{\Pi_m\star G_m}(u)$ can be bounded above with the aggregation of the Erlangization error and the discretization error.
In our developments below, we construct explicit tight bounds for each source of error.

We remark that the general formula for $\psi_{\Pi\star G}(u)$ in \cite{Bladt2015} is computational intensive and can be difficult or even infeasible to implement since it is given as an 
infinite series with terms involving products of finite dimensional matrices.  We show that for our particular model, $\psi_{\Pi\star G_m}(u)$ can be simplified down to a  manageable formula involving binomial coefficients instead of computationally expensive matrix operations.
In practice, the infinite series can be computed only up to a fine number of terms, but as
we will show, this numerical error can be controlled by selecting an appropriate distribution $\Pi$.  
Such a truncated approximation of $\psi_{\Pi\star G}(u)$ will be denoted $\widetilde{\psi}_{\Pi\star G}(u)$.
We provide explicit bounds for the  numerical error induced by truncating the infinite series.

All things considered, we contribute to the existing literature for computing ruin probabilities for the classical Cram\'er--Lundberg model by proposing a new practical numerical scheme. Our method, coupled with the bounds for the error of approximation, provides an attractive alternative for computing ruin probabilities based on a simple, yet effective idea. 

The approach described above is a further extension to  
the use of phase-type distributions for approximating
general claim size distributions \citep[cf.][]{Neuts1975,Latouche1999,Asmussen2003}.
Several attempts to approximate the probability of ruin for Cram\'er--Lundberg model have been made {(see \cite{VatamidouAdanVlasiouZwart2012} and references therein)}.
{A recent and similar approach can be found in \cite{SantanaGonzalezRincon2016} which uses discretization and Erlangizations argument as its backbone.  We emphasise here that we address the problem of finding the probability of ruin differently. Firstly, we propose to directly approximate the integrated tail distribution instead of the claim size distribution. This will yield far more accurate approximations of the probability of ruin. 
Secondly, since we investigate the discretization and the Erlangization part separately, we are able to provide tight error bounds for our approximation method. This will prove to be helpful in challenging examples such as the one presented here: the heavy-traffic Cram\'er--Lundberg model with Pareto distributed claims. Lastly, each approximation of ours is based on a mixture of Erlang distributions of fixed order, while the approach in \cite{SantanaGonzalezRincon2016} is based on a mixture of Erlang distributions of increasing order. 
By keeping the order of the Erlang distribution in the mixture fixed, we can smartly allocate more computational
resources in the discretization part, yielding an overall better approximation.  More importantly, 
we find the use of ESM more natural because 
increasing the order of the Erlang distributions in the mixture translates in having different levels of accuracy of Erlangization at different points. The choice of having sharper Erlangization in the tail of the distribution than in the body seems arbitrary and is actually not useful tail-wise, given that the tail behavior of $\Pi\ast G_m$ is the same for each $\xi(m)\ge 1$.}


The rest of the paper is organized as the follows.
Section \ref{mysec2} provides an overview of the main concepts and methods.
In Section \ref{mysec3}, we present the methodology for constructing 
a sequence of distributions of the form $\Pi_m\star G_m$ approximating the integrated tail of a general claim 
size distribution $F$. Based on the results of \cite{Bladt2015}, we introduce a 
simplified infinite series representation of the ruin probability $\psi_{\Pi_m\star G_m}$.
In Section \ref{mysec4}, we construct the bound for the error of approximation
$\left|\psi_{\widehat F}-\psi_{\Pi\star G}\right|$.
In Section \ref{sec: Numerical error bound}, we provide a bound for the numerical
error of approximation induced by truncating the infinite series representation of
$\psi_{\Pi_m\star G_m}$. A numerical example illustrating the sharpness
of our result is given in Section \ref{sub.Implementation}.
Some conclusions are drawn in 
Section \ref{mysec7}. 

\section{Preliminaries}
\label{mysec2}
In this section we provide a summary of basic concepts needed for this paper.  In subsection
\ref{sub.PH} we introduce the family of classical 
phase-type (PH) distributions and their extensions to phase-type scale mixtures and
infinite dimensional phase-type (IDPH) distributions. We will refer to the former class of distributions
as \emph{classical} in order to make a clear distinction from the two later
classes of distributions.

In section \ref{sub.Erlangization} we introduce a systematic method for approximating nonnegative distributions
within the class of phase-type scale mixtures; such a method will be called approximation via \emph{Erlangized scale
mixtures}.  The resulting approximating distribution will be more tractable due to the special structure of 
the Erlang distribution.
\subsection{Phase-type scale mixtures}\label{sub.PH}
A phase-type (PH) distribution  corresponds 
to the distribution of the absorption time of a Markov jump 
process $\{X_t\}_{t\geq 0}$
with a finite state space $E=\{0,1,2,\cdots,p\}$. The states $\{1,2,\cdots,p\}$ 
are transient while the state $0$ is an absorbing state. Hence, phase-type 
distributions are characterized by a $p$-dimensional row vector 
$\boldsymbol{\beta}=(\beta_1,\cdots,\beta_p)$, corresponding to the 
initial probabilities of each of the transient states of the Markov jump process, and an intensity matrix 
\begin{equation*}
\mathbf{Q}=\left(\begin{array}{cc}0 &\mathbf{0}  \\
             \boldsymbol{\lambda}  & \mathbf{\Lambda} \end{array}\right).
\end{equation*}
The subintensity matrix $\mathbf\Lambda$ corresponds to the transition rates among
the transient states while the column vector $\mathbf\lambda$ corresponds 
to the exit probabilities to the absorption state.  Since 
$\bm{\lambda}=- \mathbf{\Lambda e}$, 
where $\bm{e}$ is a column vector with all elements to be $1$, 
then the pair $(\bm{\beta},\mathbf{\Lambda})$
completely characterizes the absorption distribution, the notation 
$\mbox{PH}(\boldsymbol{\beta},\boldsymbol{\Lambda})$ is reserved for such a distribution. The
density function, cumulative distribution function and expectation of 
$\mbox{PH}(\boldsymbol{\beta},\boldsymbol{\Lambda})$
are given by the following closed-form expressions given 
in terms of matrix exponentials:
\begin{equation*}
g(y)=\boldsymbol{\beta}\e^{\boldsymbol{\Lambda} y}\boldsymbol{\lambda},\quad
G(y)=1-\boldsymbol{\beta}\e^{\boldsymbol{\Lambda} y}\boldsymbol{e},\quad
\int_0^{\infty}y\dd G(y)=- \boldsymbol{\beta}\boldsymbol{\Lambda}^{-1}\boldsymbol{e}.
\end{equation*}

A particular example of PH distribution which is of interest in our later
developments is that of an Erlang distribution. 
It is simple to deduce that the Erlang distribution 
with parameters ($\lambda,m$) has a PH-representation given by the
the $m$-dimensional vector $\boldsymbol{\beta}=(1,0,\cdots,0)$  and the
$m\times m$ dimensional matrix 
\begin{equation*}
\mathbf{\Lambda}=\left( \begin{array}{cccc} -\lambda & \lambda         &    & \\
                                               &\ddots  &\ddots   & \\
                                         &     &-\lambda    &\lambda\\
                                             &          &   &-\lambda
\end{array}\right).
\end{equation*}
We denote 
$\mbox{Erlang}(\lambda,m)$.
In this paper we will be particularly interested in the
sequence of $G_m\sim\mbox{Erlang}(\xi(m),\xi(m))$ distributions with $\xi(m)\to\infty$.
These type of sequencesare associated to a methodology 
often known as Erlangization (approximation of a constant via Erlang random variables). 
Using Chebyshev inequality, it is simple to prove that 
$G_m(y)\to \Ind_{[1,\infty)}(y)$ weakly, where $\Ind$ is the indicator function.  

Next, we turn our attention to the class of phase-type scale mixture distributions
\citep{Bladt2015}.
In this paper, we introduce such a class via \emph{Mellin--Stieltjes convolution} 
\begin{equation}
\label{MS convolution}
 \Pi\star G(u):=\int_0^\infty G(u/s)d\Pi(s)=\int_0^\infty \Pi(u/s)dG(s),
\end{equation}
where $G\sim\mbox{PH}(\boldsymbol{\beta},\boldsymbol{\Lambda})$
and $\Pi$ is a proper nonnegative  distribution.

Mellin--Stieltjes convolutions  
can be interpreted in two equivalent ways.  The most common one is to interpret the distribution  $\Pi\star G$
as \emph{scaled mixture distribution}; for instance, $\int G(u/s)d\Pi(s)$ can be seen as a mixture
of the scaled distributions $G_s(u)=G(u/s)$ with scaling distribution $\Pi(s)$ (and vice versa).
However, it is often more practical to see that $\Pi\star G$ corresponds to the distribution 
of the product of two independent random variables having distributions $\Pi$ and $G$. 
Furthermore, the integrated tail of $\Pi\star G$ is given in the following proposition.

\begin{Proposition}\label{Th.Excess.Dist.FG}
 Let $\Pi$ and $G$ be independent  nonnegative distributions,
 then the integrated tail of $\Pi\star G$ is given by
  \begin{equation*}
   \widehat{\Pi\star G} =  H_\Pi\star\widehat G, 
  \end{equation*}
  where $\dd H_\Pi(s)=s\dd \Pi(s)/\mu_\Pi$ is called the \emph{moment distribution of $\Pi$}
  and $\widehat{G}$ is the integrated
  tail of $G$. We use $\mu$ to denote the expecation.
\end{Proposition}

\begin{proof}
Since the Mellin--Stieltjes convoluton of $\Pi$ and $G$ can be seen as the distribution of two independent random variables having distribution $\Pi$ and $G$, then $\mu_{\Pi\star G}=\mu_{\Pi}\mu_G$.

Observe that
\begin{align}
\widehat{\Pi\star G}(u) & = \frac{1}{\mu_\Pi\cdot\mu_G}\int_0^u \left(1-\Pi\star G(t)\right)\dd t \nonumber\\
& = \frac{1}{\mu_\Pi}\int_0^u \int_0^\infty \dfrac{1-G(t/s)}{\mu_G}\dd \Pi(s) \dd t\nonumber\\
& = \int_0^\infty  \widehat{G} (u/s) \frac{ s \dd \Pi(s)}{\mu_\Pi}\nonumber\\
& = \int_0^\infty  \widehat{G} (u/s) \dd H_\Pi(s)= H_\Pi\star\widehat G(u)\nonumber.\label{eq: Spectral decomposition}
\end{align}
\end{proof}

\begin{Remark}
\label{PH Spectral decomposition}
If $G$ is a PH distribution $G\sim\mbox{PH}(\bm{\beta},\bm{\Lambda})$, then $\widehat{G}\sim\mbox{PH}(-\bm\beta\bm\Lambda^{-1}/\mu_{G},\bm\Lambda)$ is also a PH distribution
\citep[cf.][Corollary 2.3.(b), Chapter IX]{Asmussen2010}.
\end{Remark}

The following can be seen as a particular case of Proposition \ref{Th.Excess.Dist.FG} when $G$ corresponds to the point mass at one probability measure, however, a self-contained proof is provided.

\begin{Proposition}\label{Th.Excess.Dist F}
 Let $\dd H_F(s): = { s \dd F(s)}/{\mu_F}$ be the moment distribution of $F$ 
 and $U\sim\mbox{U}(0,1)$. Then
  \begin{equation*}
    \widehat F = H_F\star U. 
  \end{equation*}
\end{Proposition}

\begin{proof}
\begin{align*}
\widehat{F}(u)  = \frac{1}{\mu_F}\int_0^u \left(1-F(t)\right)\dd t
& = \frac{1}{\mu_F}\int_0^u \int_0^\infty  \Ind_{(t,\infty)}(s)\dd F(s) \dd t\\
& = \frac{1}{\mu_F}\int_0^\infty \left\{\int_0^u  {\Ind}_{[0,s)}(t)\dd t\right\}\dd F(s)\\
& = \frac{1}{\mu_F}\int_0^\infty \left\{u\wedge s\right\}\dd F(s)\\
& = \int_0^\infty \left\{(u/s)\wedge 1\right\}\frac{s \dd F(s)}{\mu_F}= H_F\star U(u),
\end{align*}
where the second equality follows from Tonelli's theorem and from the fact that for $s,t\ge 0$, $\Ind_{(t,\infty)}(s) = \Ind_{[0, s)}(t)$.
\end{proof}

In this paper we are particularly interested in the case where $\Pi$ is a discrete distribution 
having support $\{s_i:i\in\nat\}$ with $0<s_1< s_2<\dots$ and vector of probabilities $\boldsymbol{\pi}=(\pi_1,\pi_2,\cdots)$ such that $\boldsymbol{\pi\e_{\infty}}=1$, where $\boldsymbol{\e_{\infty}}$ is an infinite dimensional column vector with all elements to be $1$.
In such a case, the distribution of $\Pi\star G$ can be written as
\begin{equation*}\label{IDPH}
(\Pi\star G)(u)=\sum_{i=1}^{\infty}G(u/s_i)\pi_i,\quad u\geq 0.
\end{equation*}
Since the scaled phase-type distributions $G(u/s_i)\sim\mbox{PH}(\boldsymbol\beta,\boldsymbol\Lambda/s_i)$ 
are PH distributions again, we choose
to call $\Pi\star G$ a \emph{phase-type scale mixture} distribution.
The class of phase-type scale mixtures was first introduced in \citep{Bladt2015}, though
they restricted themselves to distributions $\Pi$ supported over the natural numbers.  
One of the main features of the class of \emph{phase-type scale mixtures} having a nonnegative discrete scaling
distribution $\Pi$ is that it forms a subclass of the so called
\emph{infinite dimensional phase-type} (IDPH) distributions; indeed, in such a case
$\Pi\star G$ can be interpreted as the distribution of absorption time 
of a Markov jump process with one absorbing state  
and infinite number of transient states, having representation 
($\boldsymbol{\alpha}$, $\boldsymbol{T}$) where
$\boldsymbol{\alpha}=$($\boldsymbol{\pi}\otimes\boldsymbol{\beta}$), 
the Kronecker product of $\boldsymbol{\pi}$ and $\boldsymbol{\beta}$, and 
\begin{equation*}
\mathbf{T}=\left( \begin{array}{cccc} \boldsymbol{\Lambda}/s_1 & 0         &0    &\cdots\\
                                      0         &\boldsymbol{\Lambda}/s_2  &0   &\cdots \\
                                      0    &0     &\boldsymbol{\Lambda}/s_3    &\cdots\\
                                      \vdots         &\vdots          &\vdots    &\ddots
\end{array}\right).
\end{equation*}
Finally, if the underlying phase-type distribution $G$ is
Erlang, and $\Pi$ is any nonnegative discrete distribution, then we say that 
the distribution $\Pi\star G$ is an \emph{Erlangized scale mixture}.  We will discuss
more properties of this distribution in later sections.

All the classes of distributions defined above are particularly attractive for modelling purposes
in part because they are dense in the nonnegative distributions (both the class of infinite dimensional
phase-type distributions and the class of phase-type scale mixtures 
trivially inherit the dense property from classical phase-type distributions, while the proof that the class of 
Erlangized scale mixtures being dense is simple and given in the next subsection).  
The class of infinite dimensional phase-type distributions contains heavy-tailed distributions but it is
mathematically intractable.  The rest of the classes defined above remain dense,
contain both light and heavy-tailed distributions
and are more tractable from both  theoretical and computational perspectives.  
Here, we concentrate on a particular subclass of the phase-type scale mixtures defined in \cite{Bladt2015}
by narrowing such a class to Erlangized scale mixtures having scaling distribution $\Pi$ with general discrete support.

\subsection{Approximations via Erlangized scale mixtures}
\label{sub.Erlangization}

Next we present a methodology for approximating an arbitrary nonnegative
distribution $\Pi$ within the class of Erlangized scale mixtures.
The construction is simple and based on the following straightforward result.
\begin{Proposition}
\label{Erlang approximation}
Let $\Pi_m$ be a sequence of nonnegative discrete distributions such that 
$\Pi_m\to \Pi$ and $G_m\sim\text{Erlang}(\xi(m),\xi(m))$. Then
\begin{equation*}
  \Pi_m\star G_{m}{\longrightarrow}\Pi.
\end{equation*}
\end{Proposition}
\begin{proof}
  Since the sequence $G_{m}$ converges weakly to $\Ind_{[1,\infty)}$, then the result
  follows directly from an application of Slutsky's theorem \citep[cf. Theorem 7.7.1][]{Ash2000}.
\end{proof}
For convenience, we refer to this method of approximation as \emph{approximation via Erlangized scale mixtures}.
The  sequence of discrete distributions $\Pi_m$ can be seen as rough 
approximations of the nonnegative distribution $\Pi$. Since $G_{m}$ is an 
absolutely continuous distribution with respect to the Lebesgue measure, then the
Mellin--Stieltjes convolution has a \emph{smoothing}
effect over the rough approximating distributions $\Pi_m$.  
Indeed, $\Pi_m\star G_{m}$ is an absolutely continuous 
distribution with respect to the Lebesgue measure (see Figure
\ref{Fig.Erlangization}). 


\begin{center}
\begin{figure}[h]
 \caption{Comparison between a target cumulative distribution function $\widehat F$
  and its \emph{Erlangized scale mixture} approximations $\Pi_m\star G_m$.  In this example we have
  considered a Pareto distribution supported over $[0,\infty)$.
  The approximating distributions $\Pi_m$ were taken as discretizations of the target distribution $\widehat F$ 
  over some geometric progression $\{\e^{t_k}: k\in\mathbb{Z},t_k=k/m\}$, while we take $\xi(m)=m$, so $G_m\sim\mbox{Erlang}(m,m)$.}
  \label{Fig.Erlangization}
\includegraphics[width=\textwidth, angle = 0]{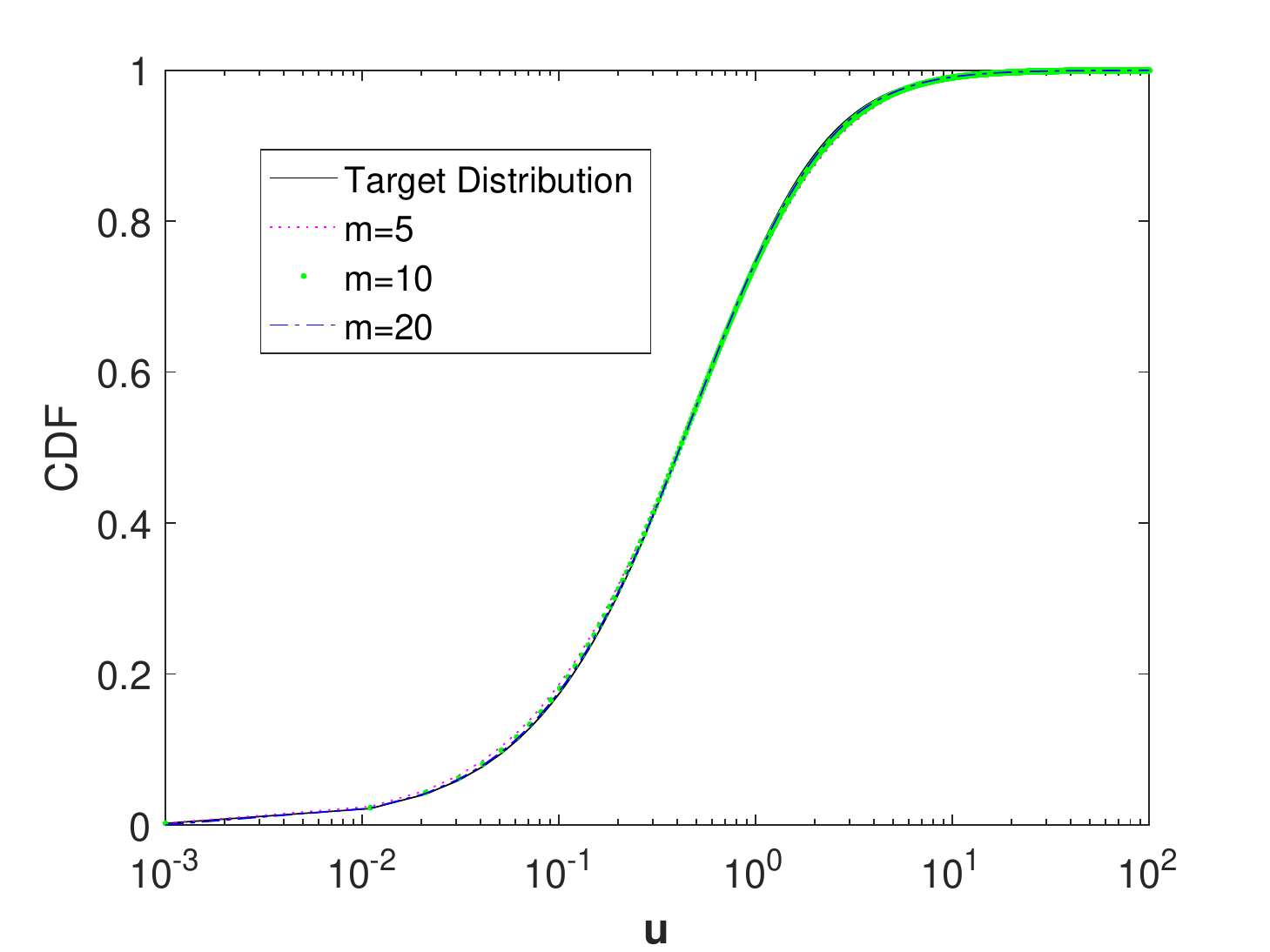}
\end{figure}
\end{center}

\section{Ruin probabilities}
\label{mysec3}

In this section we introduce a method of approximation for the
ruin probability in the Cram\'er--Lundberg risk model using Erlangized scale mixtures.
We  apply the results of \cite{Bladt2015} to obtain expressions for the ruin
probability in terms of infinite series involving operations with
finite dimensional arrays, and  exploit the simple structure of the Erlang distribution 
to obtain explicit formulas which will be free of matrix operations.
  
For constructing approximations of the ruin probability we follow two alternative approaches.  
In the first approach we approximate directly the
integrated tail distribution $\widehat F$ via Erlangized scale mixtures and is the one that
we advocate in this paper, we shall call it \emph{approximation A}. This 
straightforward approach delivers explicit formulas which are simple to write and implement;  
as we will see, the approximations obtained are very accurate.  However, the approximation 
obtained by using this approach cannot be easily related to the probability of ruin of some reserve
processes because we cannot identify an Erlangized scale mixture as the integrated tail of a
phase-type scale mixture.  Therefore, an approximating distribution for the claim sizes is not 
immediately available in this setting.  It is also required to have an explicit expression for the integrated tail
distribution $\widehat F$. 

A second approach, which is named as \emph{approximation B}, is also analysed where 
the claim size distribution is approximated with an Erlangized scale mixture. This is equivalent to approximating
the integrated tail $\widehat{F}$ with the integrated tail distribution of an Erlangized scale mixture distribution.
As we will show later, such an integrated tail distribution is in the class of phase-type distributions
so similar explicit formulas for the ruin probability are obtained.
This approach can be considered more natural but the resulting expressions are more complex and
the approximations are less accurate.
The error of approximation is bigger as a result of the \emph{amplifying} effect of integrating the tail probability
of the approximating distribution. Its implementation is more involved and the computational times are much slower when compared to the results delivered using approximation A.  

We remark that approximation B is the more commonly used, like for instance in \cite{Bladt2015} and \cite{SantanaGonzalezRincon2016}.  Thus we have included its analysis for comparison purposes.

The remaining content of this section is organised as follows: in subsection \ref{sub.Cramer} we introduce some
basic concepts of ruin probabilities in the classical Cram\'er--Lundberg risk model.
The two approximations of the ruin probability via Erlangized scale mixtures
are presented in subsection \ref{sub.RPviaErlang}.

\subsection{Ruin probability in the Cram\'er--Lundberg risk model}
\label{sub.Cramer}

We consider the classical compound Poisson risk model \citep[cf.][]{Asmussen2010}:
\begin{equation*}
R_t=u+t-\sum_{k=1}^{N_t}X_k.
\end{equation*} 
Here $u$ is the initial reserve of an insurance company, the premiums flow in 
at a rate $1$ per unit time $t$,
$X_1, X_2,\cdots$ are i.i.d. claim sizes with common distribution $F$ and mean $\mu_F$,
$\{N_t\}_{t\geq 0}$ is a Poisson process with rate $\gamma$, 
denoting the arrival of claims. So $R_t$ is a risk model for the time evolution 
of the reserve of the insurance company.
We say that ruin occurs if and only if the reserve ever drops 
below zero; we denote $\psi_{\widehat F}(u):=\inf\{R_t<0:t>0\}$.

For such a model, the well-known Pollaczek--Khinchine formula \citep[cf.][]{Asmussen2010}
implies that the ruin probability can be expressed in terms
of convolutions:
\begin{equation}
\label{PK formula}
\psi_{\widehat F}(u)=(1-\rho)\sum_{n=1}^{\infty}\rho^n\overline{\widehat{F}^{\ast n}}(u),
\end{equation}
where $\rho=\gamma\mu_F<1$ is the average claim amount per unit time, $F^{\ast n}$ denotes the $n$th-fold
convolution of $F$, $\Ftail:=1-F$ denotes the tail probability of $F$, and 
$\widehat{F}$ is the integrated tail distribution, also known as the stationary excess distribution:
\begin{equation*}
\widehat{F}(u)=\dfrac{1}{\mu_F}\int_0^u \Ftail(t)\dd t.
\end{equation*}

The calculation of ruin probability is conveniently approached via renewal theory.
The ruin probability $\psi_{\widehat{F}}(u)$ of the classical
Cram\'er--Lundberg process can be written as the probability that a terminating renewal 
process reaches level $u$.  In such a model, the distribution of the renewals is defective, and
given by $\rho\widehat F(u)$. 
In particular, if the renewals follow a defective phase-type scale mixture distribution 
with distribution $\rho \Pi\star G$ with $0<\rho<1$, then \cite{Bladt2015} derived the  the probability that the lifetime of the renewal is larger than $u$ is given by
\begin{equation}
\label{eq: Matrix exponential ruin}
\psi_{\Pi\star G}(u)=\rho\bm{\alpha}\e^{(\bm{T}+\rho\bm{t\alpha})u}\bm{e}_{\infty},
\end{equation} 
where $\bm{\alpha}=(\bm{\pi}\otimes\bm{\beta})$, $\bm T=(\bm s\times \bm I_\infty)^{-1}\otimes\bm\Lambda$
and $\bm t=-\bm T\bm e_\infty$. Here $\bm s=(s_1,s_2,\cdots)$, $\bm I$ is an identity matrix and $\bm I_{\infty}$ is that of infinite dimension. The formula above is not of practical
use because  the vectors $\boldsymbol\alpha$,  $\boldsymbol{t}$ and  the matrix 
$\boldsymbol{T}$ have infinite dimensions.   However, using the special structure of $\bm{T}$, 
they further refined the formula above and expressed $\psi_{\Pi\star G}$ as
an infinite series involving matrices and vectors of finite dimension which characterize 
the underlying distributions $\Pi$ and $G$.

Next, we obtain the explicit formula for $\psi_{\Pi\star G}(u)$ in terms of the parameters characterising the  
renewal distribution $\Pi\star G$ (equivalently the integrated tail distribution).  
This is a slight generalization of the results given in \cite{Bladt2015} who implicitly assumed that $\Pi\star G$ is the integrated tail of phase-type scale mixture distribution, so their results are given instead in terms of the parameters 
characterising the underlying claim size distribution.
{For simplicity of notation, we will write $G_m\sim$Erlang($\xi,\xi$) instead of 
Erlang($\xi(m),\xi(m)$) for the rest of the paper.}
\begin{Proposition}[\cite{Bladt2015}]
\label{renewal.PH}
Let $0<\rho<1$,
\begin{equation}
\label{BladtAlgorithmTail}
\psi_{\Pi\star G_m}(u)=\sum_{n=0}^{\infty}\kappa_n\dfrac{(\theta u/s_1)^n\e^{-\theta u/s_1}}{n!},
\end{equation}
where $\theta$ is the largest diagonal element of $-\bm{\Lambda}$ and 
\begin{align*}
\kappa_n&=\begin{cases}
  \rho,&n=0,\\[.5cm]
  \rho\left[\dfrac{ s_1}{\theta}
  \left(\sum\limits_{i=0}^{n-1}\kappa_{n-1-i}\sum\limits_{j=1}^{\infty}\dfrac{\pi_j}{s_j} B_{ij}\right) 
    +\sum\limits_{j=1}^{\infty}{\pi_j}C_{nj}\right],&n>0,
\end{cases}
\end{align*}
where
\begin{align*}
{B}_{ij}&:=\bm{\beta}(\bm{I}+(s_j\theta/s_1)^{-1}\bm{\Lambda})^i\bm{\lambda},&
{C}_{nj}&:=\bm{\beta}(\bm{I}+(s_j\theta/s_1)^{-1}\bm{\Lambda})^n\bm{e}.\\
\end{align*}
\end{Proposition}

\begin{proof}
Since $\theta$ is the largest diagonal element of $-\bm{\Lambda}$ and 
$\{s_i\}$ is an increasing sequence, then $\theta/s_1$ is the largest diagonal element of 
$-\bm T$, then from Theorem 3.1 in \cite{Bladt2015}, we have 
\begin{equation*}
\psi_{\Pi\star G_m}(u)=\sum_{n=0}^{\infty}\kappa_n\dfrac{(\theta u/s_1)^n\e^{-\theta u/s_1}}{n!},
\end{equation*}
where $\kappa_0=\rho(\bm{\pi}\otimes\bm{\beta})\bm{\e_{\infty}}=\rho\sum\limits_{i=0}^{\infty}\pi_i=\rho$, and 
\begin{equation*}
\kappa_n=\rho\left[\sum_{i=0}^{n-1}\dfrac{s_1}{\theta}(\bm{\pi}\otimes\bm{\beta})\left(\bm{I_{\infty}}+\dfrac{s_1}{\theta}\bm{T}\right)^i\bm{t}\kappa_{n-1-i}+(\bm{\pi}\otimes\bm{\beta})\left(\bm{I_{\infty}}+\dfrac{s_1}{\theta}\bm{T}\right)^n\bm{e_{\infty}}\right].
\end{equation*}
It is not difficult to see that
\begin{equation*}
(\bm{\pi}\otimes\bm{\beta})\left(\bm{I_{\infty}}+\dfrac{s_1}{\theta}\bm{T}\right)^i\bm{t}=\sum_{j=1}^{\infty}\pi_j\bm{\beta}\left(\bm{I}+\dfrac{s_1}{s_j\theta}\bm{\Lambda}\right)^i\left(-\dfrac{\bm{\Lambda\bm{e}}}{s_j}\right)=\sum_{j=1}^{\infty}\dfrac{\pi_j}{s_j}B_{ij}
\end{equation*}
and 
\begin{equation*}
(\bm{\pi}\otimes\bm{\beta})\left(\bm{I_{\infty}}+\dfrac{s_1}{\theta}\bm{T}\right)^n\bm{e_{\infty}}=
\sum_{j=1}^{\infty}\pi_j\bm{\beta}\left(\bm{I}+\dfrac{s_1}{s_j\theta}\bm{\Lambda}\right)^n\bm{e}=\sum_{j=1}^{\infty}\pi_jC_{nj},
\end{equation*}
where $B_{ij}$ and $C_{nj}$ are defined as above.
\end{proof}
Proposition \ref{renewal.PH} is to be interpreted as the probability that the lifetime
of a defective renewal process exceeds level $u$.  An interpretation in terms of the
risk process is not always possible since we may not be able to identify a claim
size distribution having integrated tail $\Pi\star G_m$. 

The result above can be seen as a (slight) generalization of  Theorem 3.1 of \cite{Bladt2015}. This can be seen from Proposition \ref{Th.Excess.Dist.FG} that shows that if the claim sizes are distributed according to an Erlangized scale mixture
$\Pi\star G_m$, 
then its integrated tail of $\Pi\star G_m$ remains in the family of phase-type scale mixtures. 
{
Using the results of Proposition \ref{Th.Excess.Dist.FG} and Remark \ref{PH Spectral decomposition}, }we recover the formula of \cite{Bladt2015}.  
\begin{Proposition}
\label{renewal.Integrated}
\begin{equation}
\label{BladtAlgorithm}
\psi_{H_\Pi\star \widehat G}(u)=\sum_{n=0}^{\infty}\kappa_n\dfrac{(\theta u/s_1)^n\e^{-\theta u/s_1}}{n!},
\end{equation}
where $\theta$ is the largest diagonal element of $-\bm{\Lambda}$ and 
\begin{align*}
\kappa_n&=\begin{cases}
  \rho,&n=0,\\[.5cm]
  \dfrac{\rho }{\mu_\Pi \mu_G}\left[\dfrac{s_1}{\theta}
  \left(\sum\limits_{i=0}^{n-1}\kappa_{n-1-i}\sum\limits_{j=1}^{\infty}{\pi_j}C_{ij}\right)
  +\sum\limits_{j=1}^{\infty}{\pi_j}s_jD_{nj}\right],&n>0,
\end{cases}
\end{align*}
where
\begin{align*}
{C}_{ij}&:=\bm{\beta}(\bm{I}+(s_j\theta/s_1)^{-1}\bm{\Lambda})^i\bm{e},&
{D}_{nj}&:=\bm{\beta}(-\bm{\Lambda})^{-1}(\bm{I}+(s_j\theta/s_1)^{-1}\bm{\Lambda})^{n}\bm{e}.
\end{align*}
\end{Proposition}

A drawback from the formulas given above is that the calculation of the quantities
$B_{ij}$, $C_{ij}$ and $D_{ij}$ is computationally expensive since these 
involve costly matrix operations. However, these expression can be simplified in our case because
because 
the subintensity matrix $\bm\Lambda$ of an
Erlang distribution can be written as a bidiagonal matrix, while the vectors denoting the 
initial distribution $\bm\beta$ and the absorption rates $\bm\lambda$ are proportional
to canonical vectors.  Hence, the resulting expressions for the terms $B_{ij}$,  $C_{ij}$
and $D_{ij}$ in Proposition \ref{renewal.PH} and Proposition \ref{renewal.Integrated} take relatively
simple forms.  These are given in the following Lemma.

\begin{Lemma}\label{lemma.matrices}
Suppose that $G_m\sim\mbox{Gamma}(\xi,\xi)$, then 
\begin{align*}
{B}_{ij}&=
    \begin{cases}\makebox[6.5cm]{$0,$}&i< \xi-1,\\[.35cm]
    \makebox[6.5cm]{${\xi}\bigg(\begin{matrix}i\\{\xi-1}\end{matrix}\bigg)\left(1-\dfrac{s_1}{s_j}\right)^{i-\xi+1}
   \left(\dfrac{s_1}{s_j}\right)^{\xi-1},$}   &i\ge \xi-1,\end{cases}\\[.25cm]
{C}_{ij}&=
    \begin{cases}\makebox[6.5cm]{$1,$}&i\le \xi-1,\\[.35cm]
    \makebox[6.5cm]{$\sum\limits_{k=0}^{\xi-1}\bigg(\begin{matrix}i\\k\end{matrix}\bigg)\left(1-\dfrac{s_1}{s_j}\right)^{i-k}
   \left(\dfrac{s_1}{s_j}\right)^k,$}   &i\ge \xi-1,\end{cases}\\[.25cm]
{D}_{ij}&=
\begin{cases}
   \makebox[6.5cm]{$1-\dfrac{i}{\xi}\dfrac{s_1}{s_j},$} &i\le \xi,\\[.35cm]
   \makebox[6.5cm]{$\sum\limits_{k=0}^{\xi-1}\dfrac{\xi-k}{\xi}\bigg(\begin{matrix}i\\k\end{matrix}\bigg)\left(1-\dfrac{s_1}{s_j}\right)^{i-k} \left(\dfrac{s_1}{s_j}\right)^k,$} &i>\xi.
 \end{cases}  
\end{align*}
\end{Lemma}
\begin{proof}
 Let $(\bm \beta,\bm\Lambda)$ be the canonical parameters of the phase-type representation of an
$\mbox{Erlang}(\xi,\xi)$ distribution (see Section 2.1), so $\theta=\xi$. Recall that
\begin{align*}
{B}_{ij}&:=\bm{\beta}(\bm{I}+(s_j\xi/s_1)^{-1}\bm{\Lambda})^i\bm{\lambda},\\
{C}_{ij}&:=\bm{\beta}(\bm{I}+(s_j\xi/s_1)^{-1}\bm{\Lambda})^i\bm{e},\\
{D}_{ij}&:=\bm{\beta}(-\bm{\Lambda})^{-1}(\bm{I}+(s_j\xi/s_1)^{-1}\bm{\Lambda})^i\bm{e}.
\end{align*}
Observe that the matrix $(\bm{I}+(s_j\xi/s_1)^{-1}\bm{\Lambda})$ is bidiagonal with all the elements
in the diagonal being equal. In particular, the $(k,\ell)$-th entry
of the $i$-th power of such a matrix is given by 
\begin{equation*}
(\bm{I}+(s_j\xi/s_1)^{-1}\bm{\Lambda})_{k\ell}^i=
 \begin{cases}
 \bigg(\begin{matrix}i\\\ell-k\end{matrix}\bigg)\left(1-\dfrac{s_1}{s_j}\right)^{i-\ell+k}\left(\dfrac{s_1}{s_j}\right)^{\ell-k}&1\le k\le\ell\le i+1 \\[.5cm]
 0&\text{otherwise}.
 \end{cases}
\end{equation*}
Therefore, $B_{ij}$ corresponds to the $(1,\xi)$-entry of the matrix $(\bm{I}+(s_j\xi/s_1)^{-1}\bm{\Lambda})$
multiplied by $\xi$. $C_{ij}$ corresponds to the sum of the elements of the first row of
$(\bm{I}+(s_j\xi/s_1)^{-1}\bm{\Lambda})$. 
For the last case, observe that
$\bm\Lambda^{-1}=-\lambda^{-1}\bm U$ where $\bm U$ is an upper triangular matrix of ones. 
Therefore, $D_{ij}$ corresponds to the sum of the elements of $(\bm{I}+(s_j\xi/s_1)^{-1}\bm{\Lambda})$
and divided by $\xi$.  $D_{ij}$ is written as the sum of all the elements in the upper diagonals divided
by $\xi$.
\end{proof}

\subsection{Ruin probability for Erlangized scale mixtures}
\label{sub.RPviaErlang}

In this subsection we specialize in approximating the ruin probability
$\psi_{\widehat F}(u)$ using  Erlangized scale mixtures.  We assume that the target
Cram\'er--Lundberg risk process has Poisson intensity $\gamma$ and claim 
size distribution $F$, so the average claim amount
per unit of time is $\rho=\gamma\mu_F$.

First, we approximate the integrated tail $\widehat F$ with an Erlangized
scale mixture $\Pi\star G_m$ where $\Pi$ is an approximating discrete
distribution of $\widehat{F}$, that is, the approach of approximation A.  The approximation for $\psi_{\widehat F}(u)$
is given next:
\begin{theorem}[Approximation A]\label{Theorem.Ruin1}
Let $\Pi$ be a  nonnegative discrete distribution supported over $\{s_i:i\in\nat\}$, 
$G_m\sim\mbox{Erlang}(\xi,\xi)$ and $\rho=\gamma\mu_F<1$. The lifetime
of a terminating renewal process having defective renewal distribution $\rho\Pi\star G_m$ 
is given by
\begin{equation*}
\psi_{\Pi\star G_m}(u)=\sum_{n=0}^{\infty}\kappa_n\dfrac{(\xi u/s_1)^n\e^{-\xi u/s_1}}{n!},
\end{equation*}
where
\begin{equation*}
\kappa_n=\begin{cases}
 \gamma\mu_F,   &     0 \leq n \le \xi-1,\\[.5cm]
     \gamma\mu_F\left[ \sum\limits_{i=\xi-1}^{n-1}\kappa_{n-1-i}
            {\mathcal{B}_{i}}	+  \mathcal{C}_{n}\right],&\xi\le n,
\end{cases}
\end{equation*}
and 
\begin{align*}
 \mathcal B_{i}&=\sum\limits_{j=1}^{\infty}\dfrac{\pi_j s_1}{s_j}
  \mbox{\emph{bin}}(\xi-1;i,s_1/s_j),
  &\mathcal C_{n}&=\sum\limits_{j=1}^{\infty}\ \pi_j
   \mbox{\emph{Bin}}(\xi-1;n,s_1/s_j),
\end{align*}
where $\mbox{\emph{bin}}(\cdot;n,p)$ and $\mbox{{\emph{Bin}}}(\cdot;n,p)$ denote the pdf and cdf 
respectively of a binomial distribution with parameters $n$ and $p$ .
\end{theorem}
\begin{proof}
The result follows by letting $\rho=\gamma\mu_F$, $\theta=\xi$, $\lambda=\xi$, 
applying Proposition \ref{renewal.PH}  and Lemma \ref{lemma.matrices} given 
in the previous subsection.
\end{proof}
We propose to use $\psi_{\Pi\star G_m}$ as an approximation of ruin probability $\psi_{\widehat F}$.  
One of the most attractive features of the result above is that because of the simple structure
of Erlangized scale mixture it is possible to rewrite the approximation of the ruin probability in
simple terms which are free of matrix operations.  In particular, the  simplified expressions
for the values of $\kappa_n$ given in terms of the binomial distribution are particularly convenient
for computational purposes.  

As stressed before, for approximation A we sacrifice the interpretation of the 
approximation $\psi_{\Pi\star G_m}$  as the ruin probability of some Cram\'er--Lundberg reserve process since
it is not possible to easily identify a distribution whose integrated tail corresponds to the Erlangized scale
mixture distribution $\Pi\star G_m$.  We also lose the interpretation of the value $\rho$ as the average claim amount per unit of time
(in the original risk process, the value of $\rho$ is selected as the product of the expected value
of an individual claim multiplied by the intensity of the Poisson process), but 
for practical computations this is easily fixed 
by simply letting $\rho=\gamma\mu_F$ where $\mu_F$ is the mean
value of the original claim sizes.  

As mentioned before, a more common and somewhat natural approach is 
to approximate the claim size distributions via Erlangized
scale mixtures, i.e. approximation B.  The following theorem provides an expression for  
approximation B of the probability of ruin $\psi_{\widehat F}$ with 
the ruin probability of a reserve process having claim sizes $\Pi\star G_m$.
This result could be useful for instance in a situation where the integrated
tail is not available and it is difficult to compute. 

Note that we have modified the intensity of the Poisson process in order to match
the average claim amount per unit of time $\rho=\gamma\mu_F$ of the original process.  This 
selection will help to demonstrate uniform convergence.  

\begin{theorem}[Approximation B]\label{Theorem.Ruin2}
Let $\Pi$ be a  nonnegative discrete distribution supported over $\{s_i:i\in\nat\}$ and 
$G_m\sim\mbox{Erlang}(\xi,\xi)$. The probability of ruin in the Cram\'er--Lundberg model
having intensity $\gamma\mu_F/\mu_\Pi$ and claim size distribution $\Pi\star G$ is given by
\begin{equation*}
\psi_{H_\Pi\star \widehat G_m}(u)=\sum_{n=0}^{\infty}\kappa_n\dfrac{(\xi u/s_1)^n\e^{-\xi u/s_1}}{n!},
\end{equation*}
where
\begin{equation*}
\kappa_n=\begin{cases}
 \gamma\mu_F,   &   n=0,\\[.5cm]
    (\gamma\mu_F-1)\left(1+\dfrac{\gamma \mu_F s_1 }{{  \mu_{\Pi}} \xi}\right)^n+1,  &   1 \leq n \leq \xi,\\[.5cm]
     \dfrac{ \gamma \mu_F s_1}{{ \mu_{\Pi}} \xi }\sum\limits_{i=0}^{n-1}\kappa_{n-1-i}
            \mathcal{C}_{i}	+  \dfrac{ \gamma \mu_F }{\mu_{\Pi}}
            \mathcal{D}_{n},&\xi< n.
\end{cases}
\end{equation*}
and
\begin{align*}
 \mathcal C_{i}&=\sum\limits_{j=1}^{\infty}\pi_j\mbox{\emph{Bin}}(\xi-1;i,s_1/s_j),
 &\mathcal D_{n}&=\sum\limits_{j=1}^{\infty}\ \pi_j s_j\sum\limits_{k=0}^{\xi-1}\dfrac{\xi-k}\xi\mbox{\emph{bin}}(k;n,s_1/s_j).
\end{align*}

\end{theorem}

\begin{proof}
 Let $\theta=\xi$ and $\lambda=\xi$. If $1\le n \le \xi$, then  from Proposition \ref{renewal.Integrated} and Lemma \ref{lemma.matrices}.  
 we have that
\begin{align*}
\kappa_n&= \dfrac{ \gamma \mu_F s_1}{{ \mu_{\Pi}} \xi }\sum_{i=0}^{n-1}\kappa_{n-1-i}
+ \dfrac{ \gamma \mu_F }{{ \mu_{\Pi}}}\sum\limits_{j=1}^{\infty}\pi_js_j\left(1-\frac{n}{\xi}\frac{s_1}{s_j}\right),\\
&= \dfrac{ \gamma \mu_F s_1}{{ \mu_{\Pi}} \xi }\sum_{i=0}^{n-1}\kappa_{n-1-i}
+\gamma\mu_{F}- \dfrac{ \gamma \mu_F }{{ \mu_{\Pi}} \xi }\sum_{j=1}^{\infty}s_j\pi_jn\dfrac{s_1}{s_j}\\
&= \dfrac{ \gamma \mu_F s_1}{{ \mu_{\Pi}} \xi }\left(\sum_{i=0}^{n-1}\kappa_i-n\right)+\gamma\mu_{F}.
\end{align*}
Then by induction, we can get for $1\leq n \leq \xi$,
\begin{equation*}
\kappa_n=(\gamma\mu_F-1)\left(1+ \dfrac{ \gamma \mu_F s_1}{{ \mu_{\Pi}} \xi }\right)^n+1.
\end{equation*}
The cases $n=0$ and $\xi<n$ follow directly from applying Proposition \ref{renewal.Integrated} and Lemma \ref{lemma.matrices}.
\end{proof}

\section{Error bounds for the ruin probability}
\label{mysec4}

In this section we will assess the accuracy of the two proposed
approximations for the ruin probability. We will do so by providing bounds 
for the error of approximation.
We identify two sources of error. The first source is due
to the Mellin--Stieltjes convolution with the Erlang distribution; we will call this
the \emph{Erlangization error}. The second source of error is due to the 
approximation of the integrated tail $\widehat F$ (via $\Pi$ in the first case, and
via $H_\Pi$ in the second case); we will refer to this as the
\emph{discretization error}.  
For the case of approximation A  in Theorem \ref{Theorem.Ruin1}  we can use the triangle inequality to bound the overall error with the aggregation of the two types of errors, that is
\begin{equation*}
 \left|\psi_{\widehat F}(u)-\psi_{\Pi\star G_m}(u)\right|\le 
  \left|\psi_{\widehat F}(u)-\psi_{\widehat F\star G_m}(u)\right|+ 
 \left|\psi_{\widehat F\star G_m}(u)-\psi_{\Pi\star G_m}(u)\right|.
\end{equation*} 
For  approximation B in Theorem \ref{Theorem.Ruin2} we have an analogous bound
\begin{equation*}
 \left|\psi_{\widehat F}(u)-\psi_{H_\Pi\star \widehat G_m}(u)\right|\le 
  \left|\psi_{\widehat F}(u)-\psi_{H_F\star \widehat G_m}(u)\right|+ 
 \left|\psi_{H_F\star \widehat G_m}(u)-\psi_{H_\Pi\star \widehat G_m}(u)\right|.
\end{equation*} 

We will rely on the Pollaczek--Khinchine formula \eqref{PK formula} for the construction of the bounds.
Recall that the formula above is interpreted as the probability that a terminating renewal process
having defective renewal probability $\rho\widehat{F}(\cdot)$ will reach level $u$ before
terminating. 
In our two approximations of $\psi_{\widehat F}$, we have selected the value of $\rho=\gamma\mu_F$
so we can write the errors of approximation in terms of the differences between
the convolutions of the integrated tail exclusively.  For instance, the error of Erlangization
in  approximation A is given by
\begin{equation}
\label{error bond erlangization}
 \left|\psi_{\widehat F}(u)-\psi_{\widehat F\star G_m}(u)\right|
  =\left|\sum_{n=1}^{\infty}(1-\rho)\rho^n\left(\overline{\widehat{F}^{\ast n}}(u)
   -\overline{\widehat F\star G_m^{\ast n}}(u)\right)\right|.
\end{equation}
Note that $n=0$ in the above series is equal to zero.

For our  approximation B, it is noted that setting the parameter $\rho=\gamma\mu_F$ is equivalent
to calculating the ruin probability for a risk process having integrated claim sizes distributed according to 
$H_\Pi\star \widehat G_m$ while the intensity of the Poisson process is changed to $\gamma\mu_F/\mu_\Pi$.
With such an adjustment, it is possible to write both the Erlangization and discretization errors in
terms of differences of higher order convolutions as given above.

We will divide this section in three parts. In subsection \ref{sec.Vatamidou} we refine an existing
bound introduced in \cite{VatamidouAdanVlasiouZwart2014} for the error of approximation of the ruin probability.
This refined result will be used in the construction of bounds for the
error of discretization.
In subsections \ref{sec.error.app1} and \ref{sec.error.app2} we
provide bounds for the errors for each of the two approximations proposed.

\subsection{General bounds for the error of approximation}
\label{sec.Vatamidou}

The following Theorem provides a refined bound for the error of approximation for the ruin probability provided
by \cite{VatamidouAdanVlasiouZwart2014}.

\begin{theorem}\label{Th.Vatamidou} For any distributions with positive support $\widehat{F_1}$ and $\widehat{F_2}$ and fixed $u>0$, we have that
\begin{align*}
|\psi_{\widehat{F_1}}(u)-\psi_{\widehat{F_2}}(u)|& \le  \sup_{s<u}\{|\widehat{F_1}(s) - \widehat{F_2}(s)|\} 
\frac{(1-\rho) \rho}{(1-\rho \widehat{F_1}(u))(1-\rho \widehat{F_2}(u))}. \\
\end{align*}
\end{theorem}
\begin{proof}
We claim that for any $n\ge 1$,
\begin{equation}\label{eq:vatamidoushaper}
\sup_{s<u}\{|{\widehat{F_1}}^{*n}(s) - \widehat{F_2}^{*n}(s)|\le \sup_{s<u}\{|\widehat{F_1}(s) - \widehat{F_2}(s)|\} \sum_{i=0}^{n-1}\widehat{F_1}^i(u)
\widehat{F_2}^{n-1-i}(u).
\end{equation}
Let us prove it by induction. It is clearly valid for $n=1$. Let us assume that it is valid for some $n\ge 1$. Then
\begin{align*}
&\sup_{s<u}\{|\widehat{F_1}^{*n+1}(s) - \widehat{F_2}^{*n+1}(s)|\}\\
 & \quad= \sup_{s<u}\{|\widehat{F_1}^{*n+1}(s) - \widehat{F_1}^{*n}*\widehat{F_2}(s) +  \widehat{F_1}^{*n}*\widehat{F_2}(s) - \widehat{F_2}^{*n+1}(s)|\}\\
& \quad\le\sup_{s<u}\{|\widehat{F_1}^{*n+1}(s) - \widehat{F_1}^{*n}*\widehat{F_2}(s)|\} +  \sup_{s<u}\{| \widehat{F_1}^{*n}*\widehat{F_2}(s) - \widehat{F_2}^{*n+1}(s)|\}.
\end{align*}
Clearly,
\begin{align}
\sup_{s<u}\{|\widehat{F_1}^{*n+1}(s) - \widehat{F_1}^{*n}*\widehat{F_2}(s)|\} & \le \sup_{s<u}\left\{\int_0^s|\widehat{F_1}(r) - \widehat{F_2}(r)|\dd \widehat{F_1}^{*n}(r)\right\}\nonumber\\
& \le \sup_{s<u}\left\{\int_0^s\sup_{l<u}\{|\widehat{F_1}(l) - \widehat{F_2}(l)|\}\dd \widehat{F_1}^{*n}(r)\right\}\nonumber\\
&= \sup_{l<u}\{|\widehat{F_1}(l) - \widehat{F_2}(l)|\} \sup_{s<u}\left\{\int_0^s\dd \widehat{F_1}^{*n}(r)\right\}\nonumber\\
&= \sup_{l<u}\{|\widehat{F_1}(l) - \widehat{F_2}(l)|\} \widehat{F_1}^{*n}(u)\nonumber\\
&\le \sup_{l<u}\{|\widehat{F_1}(l) - \widehat{F_2}(l)|\} \widehat{F_1}^n(u).\label{eq:sharpvatpart1}
\end{align}
In the last step we have used that $\widehat F^{\ast n}(u)$ corresponds to the probability
of an event where the sum of $n$ i.i.d. random variables is smaller equal than $u$ while
$\widehat{F}^n(u)$ corresponds to the probability of the maximum of i.i.d. random variables is 
smaller equal than $u$; if the random variables are nonnegative then the probability of the 
sum is clearly smaller  than the probability of the maximum.
Using the hypothesis induction we have that
\begin{align}
\sup_{s<u}\{| \widehat{F_1}^{*n}*\widehat{F_2}(s) - \widehat{F_2}^{*n+1}(s)|\} & \le \sup_{s<u}\left\{\int_0^s|\widehat{F_1}^{*n}(r) - \widehat{F_2}^{*n}(r)|\dd \widehat{F_2}(r)\right\}\nonumber\\
& \le \sup_{s<u}\left\{\int_0^s\sup_{l<u}\{|\widehat{F_1}^{*n}(l) - \widehat{F_2}^{*n}(l)|\}\dd \widehat{F_2}(r)\right\}\nonumber\\
& = \sup_{l<u}\{|\widehat{F_1}^{*n}(l) - \widehat{F_2}^{*n}(l)|\} \sup_{s<u}\left\{\int_0^s\dd \widehat{F_2}(r)\right\}\nonumber\\
& \le \left(\sup_{s<u}\{|\widehat{F_1}(s) - \widehat{F_2}(s)|\} \sum_{i=0}^{n-1}\widehat{F_1}^i(u)\widehat{F_2}^{n-1-i}(u)\right)\widehat{F_2}(u)\nonumber\\
& = \sup_{s<u}\{|\widehat{F_1}(s) - \widehat{F_2}(s)|\} \sum_{i=0}^{n-1}\widehat{F_1}^i(u)\widehat{F_2}^{n-i}(u).\label{eq:sharpvatpart2}
\end{align}
Summing (\ref{eq:sharpvatpart1}) and (\ref{eq:sharpvatpart2}), we get that
\begin{align*}
\sup_{s<u}\{|\widehat{F_1}^{*n+1}(s) - \widehat{F_2}^{*n+1}(s)|\}\le \sup_{s<u}\{|\widehat{F_1}(s) - \widehat{F_2}(s)|\} \sum_{i=0}^{n}\widehat{F_1}^i(u)\widehat{F_2}^{n-i}(u),
\end{align*}
so that formula (\ref{eq:vatamidoushaper}) is valid for all $n\ge 1$.
Finally,
\begin{align*}
|\psi_{\widehat{F_1}}(u)-\psi_{\widehat{F_2}}(u)|& \le \sum_{n=1}^\infty (1-\rho)\rho^n |{\widehat{F_1}}^{*n}(u) - {\widehat{F_2}}^{*n}(u)|\\
& \le \sup_{s<u}\{|{\widehat{F_1}}(s) - {\widehat{F_2}}(s)|\}(1-\rho)  \sum_{n={1}}^\infty \rho^n \sum_{i=0}^{n-1}{\widehat{F_1}}^i(u){\widehat{F_2}}^{n-1-i}(u)\\
& = \sup_{s<u}\{|{\widehat{F_1}}(s) - {\widehat{F_2}}(s)|\} (1-\rho)  \sum_{i=0}^{\infty}\sum_{n=i+1}^\infty \rho^n {\widehat{F_1}}^i(u){\widehat{F_2}}^{n-1-i}(u)\\
& = \sup_{s<u}\{|{\widehat{F_1}}(s) - {\widehat{F_2}}(s)|\}(1-\rho)  \sum_{i=0}^{\infty}\sum_{n=0}^\infty \rho^{n+i+1} {\widehat{F_1}}^i(u){\widehat{F_2}}^{n}(u)\\
& = \sup_{s<u}\{|{\widehat{F_1}}(s) - {\widehat{F_2}}(s)|\} (1-\rho) \rho \sum_{i=0}^{\infty}\rho^i {\widehat{F_1}}^i(u) \sum_{n=0}^\infty \rho^{n} {\widehat{F_2}}^{n}(u)\\
& = \sup_{s<u}\{|{\widehat{F_1}}(s) - {\widehat{F_2}}(s)|\} (1-\rho) \rho \frac{1}{1-\rho {\widehat{F_1}}(u)} \frac{1}{1-\rho {\widehat{F_2}}(u)}\\
& = \sup_{s<u}\{|{\widehat{F_1}}(s) - {\widehat{F_2}}(s)|\}  \frac{(1-\rho) \rho}{(1-\rho {\widehat{F_1}}(u))(1-\rho {\widehat{F_2}}(u))}.
\end{align*}
\end{proof}
We remark that the bound given above is a refinement of the result 
obtained in \cite{VatamidouAdanVlasiouZwart2014}:  
The construction of our bound is based on the inequality \eqref{eq:sharpvatpart2} and given by
\begin{align*}
\sup_{s<u}\{| \widehat{F_1}^{*n}*\widehat{F_2}(s) - \widehat{F_2}^{*n+1}(s)|\} 
 \le \sup_{s<u}\{|\widehat{F_1}(s) - \widehat{F_2}(s)|\} \sum_{i=0}^{n-1}\widehat{F_1}^i(u)\widehat{F_2}^{n-i}(u).
\end{align*}
The expression on the right hand side takes values in $(0,1)$ for all values of $n$.
In contrast, the quantity used in \cite{VatamidouAdanVlasiouZwart2014} to bound the expression
in the left hand side is $n\widehat{F}(u)$, which goes to infinity as $n\to\infty$.
We remark however, that the final bound for the error term 
proposed there remains bounded.
A comparison of the two bounds reveals that the one suggested above 
improves \cite{VatamidouAdanVlasiouZwart2014}'s bound by a factor of
\begin{equation*}
  \frac{(1-\rho)^2}{(1-\rho \widehat{F_1}(u))(1-\rho \widehat{F_2}(u))}\le 1.
\end{equation*}

\subsection{Error bounds for $\psi_{\Pi\star G_m}$}
\label{sec.error.app1}

This subsection is dedicated to the construction of the bounds for 
approximation A suggested in Theorem \ref{Theorem.Ruin1}.    
\subsubsection{Bounds for the Erlangization error of $\psi_{\Pi\star G}$}
A bound for the Erlangization error is constructed throughout the following results.


\begin{Lemma}\label{Lemma.From.Hell}
Let $\{\mathcal{A}_k:k\in\nat\}$ be an decreasing collection of closed intervals in $\RL^+$, so 
$\mathcal{A}_k=[a_k,b_k]$ and $\mathcal{A}_{k+1}\subset\mathcal{A}_{k}$. If $\mathcal{A}_0=[0,\infty]$ and
$\mathcal{A}_k\searrow\{1\}$ then
\begin{align*}
 \sup_{\ell\le u} \left|\widehat F(\ell)-\widehat F\star G_m(\ell)\right|
\le \sum_{k=0}^\infty\sup_{\ell<u}(\widehat F_{b_k}(\ell)-\widehat F_{a_k}(\ell))(G_m(\mathcal{A}_k)-G_m(\mathcal{A}_{k+1})),
\end{align*}
where $G_m(\mathcal{A}_k):=G_m(b_k)-G_m(a_k)$.  
\end{Lemma}
\begin{proof}
\begin{align*}
 \sup_{\ell\le u} \left|\widehat F(\ell)-\widehat F\star G_m(\ell)\right|
 &\le \sup_{\ell<u}\left|\sum_{k=0}^\infty\left[\widehat F(\ell)\int\limits_{\mathcal{A}_k/\mathcal{A}_{k+1}}\dd G_m(s)
     -\int\limits_{\mathcal{A}_k/\mathcal{A}_{k+1}} \widehat F(\ell/s)\dd G_m(s)\right]\right|\\
 &=\sum_{k=0}^\infty \sup_{\ell<u}\left|\int\limits_{\mathcal{A}_k/\mathcal{A}_{k+1}}[\widehat F(\ell)-\widehat F(\ell/s)]\dd G_m(s)\right|\\
 &\le \sum_{k=0}^\infty\sup_{\ell<u}(\widehat F(\ell/b_k)-\widehat F(\ell/a_k))(G_m(\mathcal{A}_k)-G_m(\mathcal{A}_{k+1}))\\
 &\le \sum_{k=0}^\infty\sup_{\ell<u}(\widehat F_{b_k}(\ell)-\widehat F_{a_k}(\ell))(G_m(\mathcal{A}_k)-G_m(\mathcal{A}_{k+1})).
\end{align*}
\end{proof}

An upper bound for the Erlangization error is given next.

\begin{theorem}\label{Erlagization.Bound.A}
Let $\{\mathcal{A}_k:k\in\nat\}$ be a sequence as defined in Lemma \ref{Lemma.From.Hell}. Then
\begin{align*}
\left| \psi_{\widehat F}(u) - \psi_{\widehat F\star G_m}(u) \right| 
\le    \frac{ \rho}{(1-\rho \widehat{F}(u))}
 \sum_{k=0}^\infty\sup_{\ell<u}(\widehat F_{b_k}(\ell)-\widehat F_{a_k}(\ell))
   (G_m(\mathcal{A}_k)-G_m(\mathcal{A}_{k+1})).
\end{align*}
Moreover, if $\widehat F$ is absolutely continuous with
bounded density then
$\psi_{\widehat F}(u) \to \psi_{\widehat F\star G_m}(u)$ uniformly as $\xi(m)\to\infty$.
\end{theorem}
In our numerical experiments we found that it is enough to take a finite number $K$ of 
sets $\mathcal{A}_1,\dots,\mathcal{}A_K$ to obtain a usable numerical bound.  
This is equivalent to take $\mathcal{A}_k=\{1\}$ for all $k\ge K$ in the Theorem above.

\begin{proof}
The proof follows from Theorem \ref{Th.Vatamidou}, Lemma \ref{Lemma.From.Hell} and the following
observation
\begin{align*}
 \frac{1}{1-\rho\widehat F\star G_m(u)}\le \frac{1}{1-\rho}, \qquad\forall u>0.
\end{align*}
To prove uniform convergence we simply note that the expression above can be further bounded
above by
\begin{align}\label{Sharp.Bound}
  \left| \psi_{\widehat F}(u) - \psi_{\widehat F\star G_m}(u) \right| 
  \le \frac{\rho}{1-\rho} \sum_{k=0}^\infty{\sup_{\ell>0}(\widehat F_{b_k}(\ell)-\widehat F_{a_k}(\ell)) 
      \left(G_m(\mathcal{A}_{k})-G_m(\mathcal{A}_{k+1})\right)}.
\end{align}
Notice that if $\widehat F$ is an absolutely continuous distribution with a bounded density, then for
any sequence of nonempty sets such that $\mathcal{A}_k\searrow\{1\}$, it holds that
for every $\epsilon>0$ we can find $k_0\in\nat$ such that $\sup_{\ell>0}(\widehat F_{b_k}(\ell)-\widehat F_{a_k}(\ell))<\epsilon (1-\rho)/2\rho$
for all $k>k_0$.  Similarly, we can find $\xi(m_0)\in\nat$ large enough such that $1-G_m(A_{k+1})\le \epsilon (1-\rho)/2\rho$.
Putting together this results we obtain that for all $k\ge k_0$ and $m\ge m_0$
\begin{align*}\label{Sharp.Bound}
  \left| \psi_{\widehat F}(u) - \psi_{\widehat F\star G_m}(u) \right| 
  \le \frac{\rho}{1-\rho}\left[ {\sup_{\ell>0}(\widehat F_{b_k}(\ell)-\widehat F_{a_k}(\ell))+ 
      \left(1-G_m(\mathcal{A}_{k+1})\right)}\right] = \epsilon.
\end{align*}
Hence, uniform convergence follows.
\end{proof}

\subsubsection{Bounds for the discretization error of $\psi_{\Pi\star G}$}
\label{subsubsec.Discrete.I}

Next, we address the construction of a bound for the discretization error:

\begin{equation*}
 \left|\psi_{\widehat F\star G_m}(u)-\psi_{\Pi\star G_m}(u)\right|.
\end{equation*} 

The following Theorem makes use of our refinement of \cite{VatamidouAdanVlasiouZwart2014}'s bound
for the construction of an upper bound for the discretization error.

\begin{theorem}\label{Theorem.error.discret.I}
 Let 
 \begin{equation*}
   \eta:=\sup_{0\le s\le u}\{|\widehat F\star  G_m(s) - \Pi\star  G_m(s)|\}
 \end{equation*}
 then for all $0<\delta<\infty$ it holds that
\begin{align*}
\left|\psi_{\widehat F\star  G_m}(u)-\psi_{\Pi\star G_m}(u)\right|
  & \le    \frac{\eta(1-\rho) \rho}{\left(1-\rho (\widehat F(u/\delta)+ G_m(\delta))\right)
   \left(1-\rho (\Pi(u/\delta)+ G_m(\delta))\right)}.
\end{align*}
\end{theorem}
The bound above decreases as  $\Pi$ gets close to $\widehat F$;
this is reflected in the value of $\eta$.  The bound will become smaller as long
as terms $\widehat{F}(u/\delta)+ G_m(\delta)$ and $\Pi(u/\delta)+ G_m(\delta)$
in the denominator become bigger.  The value of $\delta$ minimizing this bound can
be easily found numerically.  

\begin{proof}
The result follows from observing that 
\begin{align*}
 \widehat F\star G_m(u)
 = \int_0^\delta \widehat F(u/s)\dd G_m(s)+\int_\delta^\infty \widehat F(u/s)\dd G_m(s)\le 
   \widehat F(u/\delta) +  G_m(\delta).
\end{align*}
We just apply our refinement of \cite{VatamidouAdanVlasiouZwart2014}'s bound provided in Theorem \ref{Th.Vatamidou}. 
\begin{align*}
\left|\psi_{\widehat F\star G_m}(u)-\psi_{\Pi\star  G_m}(u)\right|
  & \le    \frac{\eta(1-\rho) \rho}{(1-\rho\widehat F\star  G_m(u))(1-\rho \Pi\star  G_m(u))}. 
\end{align*}
A lower bound for $\Pi\star G_m(u)$ can be found in an analogous way.
\end{proof}

The last step in the construction of an upper bound for the discretization error
is finding an upper bound for $\eta=\sup_{0\le s\le u}|\widehat F\star{G}_m(s)-\Pi\star{G}_m(s)|$.
We suggest a bound in the following Proposition.

\begin{Proposition}\label{Prop.error.discrete.I}
 Let $0<\delta<\infty$, then
\begin{align*}
   \sup_{0\le s\le u/\delta}|\widehat F\star G_m(s) - \Pi\star G_m(s)| 
    \le \eta(\delta),
   \end{align*}
  where
  \begin{align*}
   \eta(\delta)=\sup_{u/\delta\le s<\infty} \left|{\widehat{F}}(s)-{\Pi}(s)\right|G_m(\delta)
       +\sup_{0<s\le u/\delta}\left|\widehat{F}(s)-\Pi(s)\right|\overline G_m(\delta).
\end{align*}   
\end{Proposition}
\begin{proof}
 \begin{align*}
  \left|\widehat {F}\star  G_m(u) - {\Pi}\star  G_m(u)\right|&=
   \left|\int_0^\infty \widehat F(u/s)\dd G_m(s)-\int_0^\infty \Pi(u/s)\dd G_m(s)\right|\nonumber\\
   &\le \int_0^\infty \left|\widehat F(u/s)-\Pi(u/s)\right|\dd G_m(s)\nonumber\\
   &\le \int_0^\delta \left|\widehat F(u/s)-\Pi(u/s)\right|\dd G_m(s)
    +\int_{\delta}^\infty \left|\widehat F(u/s)-\Pi(u/s)\right|\dd G_m(s)\nonumber\\
   &\le \sup_{u/\delta\le s<\infty} \left| \widehat F(s)-\Pi(s)\right| G_m(\delta)
     +\sup_{0<s\le u/\delta}\left|\widehat F(s)-\Pi(s)\right|\overline G_m(\delta).
 \end{align*}
\end{proof}
In practice, we would select a value of $\delta$ which minimizes the upper bound $\eta(\delta)$.
Notice, that if the tail probability of $F$ is well approximated by $\Pi$, then
the error bound will in general decrease.  This suggests that 
$\Pi$  should provide a good approximation of $\widehat F$ particularly in the tail in
order to reduce effectively the error of approximation.

\subsection{Error bounds for $\psi_{H_\Pi\star \widehat G_m}$}
\label{sec.error.app2}

Next we turn our attention to  approximation B of the ruin probability when
the claim size distribution $F$ is approximated via Erlangized scale mixtures.
We remark that the bounds presented in this section are simple and sufficient to show uniform
convergence. However, these bounds are too rough for practical
purposes.  A set of more refined bounds can be obtained but their construction
and expressions are more complicated, so these have been relegated to
the appendix.

\subsubsection{Bounds for the Erlangization error of $\psi_{H_\Pi\star \widehat G}$}
\label{sec: First error}

The following theorem provides a first bound for the Erlangization error of 
the approximation $\psi_{H_F\star \widehat G_m}$.  A tighter bound for the Erlangization error
can be found in the Appendix.
\begin{theorem}
\label{First error bound}
\[\left| \psi_{H_F\star U}(u) - \psi_{H_F\star \widehat G_m}(u) \right| \le 
  \frac{2\rho\epsilon_m}{1-\rho(1-\epsilon_m)}\le \frac{\rho}{1-\rho}\sqrt{\frac{2}{\pi m}},\]
where $\epsilon_m$ is defined as in Lemma \ref{epsilon}.
\end{theorem}
\begin{proof}
Since $G_m(s)\rightarrow{{\Ind}}_{[1, \infty )}(s)$
for all $s\neq1$  so 
\[\widehat{g}_m(s):=\frac{\dd}{\dd s} \widehat{G}_m(s) 
 = 1 - G_m(s)\rightarrow{{\Ind}}_{[0, 1)}(s),\qquad\forall s\neq1.\]
Let $\{X_n'\}$ be a sequence of independent 
and identically $H_F$ distributed random variables. Then, by Propositions \ref{Th.Excess.Dist.FG} - \ref{Th.Excess.Dist F},
\begin{align*}
& \left| \widehat F^{*n}(u) - \widehat{F\star G}^{*n}(u)\right|
=\left| (H_F\star U)^{*n}(u) - (H_F\star \widehat G_m)^{*n}(u)\right|\nonumber\\
&\qquad \le \idotsint\limits_{\RL^n} \mathbb{P}(s_1 X_1' + \dots + s_n X_n'\le u ) \left| \prod_{i=1}^n {\Ind}_{[0,1)}(s_i)-\prod_{i=1}^n \widehat{g}_m(s_i)    \right| \dd s_1 \dots \dd s_n\\
&\qquad \le \idotsint\limits_{\RL^n} \left| \prod_{i=1}^n {\Ind}_{[0,1)}(s_i)-\prod_{i=1}^n \widehat{g}_m(s_i)    \right| \dd s_1 \dots \dd s_n.
\end{align*}
That the last integral is bounded by $2(1-(1-\epsilon_m)^n)$ follows from Corollary \ref{Cor.Int} in the Appendix.
Therefore, we have that
\begin{align*}
\left| \psi_F(u) - \psi_{F\star G_m}(u) \right| 
& \le \sum_{n=1}^\infty (1-\rho)\rho^n\left| \widehat{F}^{*n}(u) - (\widehat{F\star G}_{m})^{*n}(u)\right|\\
& \le \sum_{n=1}^\infty (1-\rho)\rho^n 2\left(1-(1-\epsilon_m)^n\right)\\
& = 1-\frac{1-\rho}{1-\rho(1-\epsilon_m)}
=\frac{2\rho\epsilon_m}{1-\rho(1-\epsilon_m)}.\label{eq:completearea}
\end{align*}

Lemma \ref{epsilon} provides an explicit bound for $\epsilon_m$.
\end{proof}
The following result provides with an explicit expression useful for obtaining 
the integrated distance between the survival function  $1-G_m$ and the density of a $\mbox{U}(0,1)$ distribution.  That is
\begin{equation*}
 \int_0^\infty\Big|(1-G_m(s))-\Ind_{0,1}(s)\Big|\dd s.
\end{equation*}

\begin{Lemma}\label{epsilon}
 \begin{equation*}
  \epsilon_m=\int_0^1 G_m(s)ds =\int_1^\infty \left(1-G_m(s)\right)ds=\e^{-\xi}\dfrac{\xi^{\xi}}{\xi!}\le (2\pi \xi)^{-\frac{1}{2}}.
 \end{equation*}
\end{Lemma}

\begin{proof}
Firstly observe that $\mu_{G_m}=1$, it follows that $1-G_m$ is the density of the integrated tail distribution $\widehat G_m$.
Hence,
\[\int_1^\infty \left(1-G_m (s)\right)\dd s=1- \int_0^1  \left(1-G_m (s)\right)dd s =
\int_0^1 G_m (s)\dd s =\epsilon_m, \]
and the second equality follows.  For the third equality we have that
\begin{align*}
\epsilon_m&=\int_0^1G_m(s)\dd s
=\int_0^1\left(1-\sum_{n=0}^{\xi-1}\frac{1}{n!}\e^{-\xi s}(\xi s)^n\right)\dd s
=1-\sum_{n=0}^{\xi-1}\frac{1}{n!}\int_0^1\e^{-\xi s}(\xi s)^n\dd s\\
&=1-\sum_{n=0}^{\xi-1}\frac{1}{n!}
\left(n!\xi^{-1}-\e^{-\xi}\sum_{k=0}^{n}\dfrac{n!\xi^{k-1}}{k!}\right)
=\e^{-\xi}\sum_{n=0}^{\xi-1}\sum_{k=0}^{n}\frac{\xi^{k-1}}{k!}\\
&=\e^{-\xi}\sum_{k=0}^{\xi-1}(\xi-k)\frac{\xi^{k-1}}{k!}
=\e^{-\xi}\left(\sum_{k=0}^{\xi-1}\frac{\xi^{k}}{k!}-\sum_{k=1}^{\xi-2}k\frac{\xi^{k}}{k!}\right)
=\e^{-\xi}\dfrac{\xi^{\xi}}{\xi!}.
\end{align*}
Finally, an application of Stirling's formula 
$\xi!>\sqrt{2\pi}\xi^{\xi+\frac{1}{2}}\e^{-\xi}$ yields $\epsilon_m<(2\pi\xi)^{-\frac{1}{2}}$.
\end{proof}

Note that the bound for the error provided above only depends on 
the parameter of the Erlang distribution $\xi$ and the average claim amount per unit of time $\rho$. 
This bound does not depend on the initial reserve $u$, nor the underlying claim size distribution $F$,
so $\psi_{\widehat F\star \widehat G_m}$ converges uniformly to $\psi_{\widehat F}$.
However, in practice this bound is too rough and not useful for practical purposes.  In Theorem 
\ref{Refinement.First.Bound} we provide a refinement of the bound above. The refined bound proposed in there
no longer has a simple form but in return it is much sharper and more useful
for practical purposes.

\subsubsection{Bounds for the discretization error of $\psi_{H_\Pi\star \widehat G}$}\label{sec: Second error}
Finally, we address the construction of a bound for the discretization error.
The next two results are analogous to the ones in subsection \ref{subsubsec.Discrete.I}
and presented without proof.
\begin{theorem}
\label{second error bound}
 Let 
 \begin{equation*}
   \eta:=\sup_{0\le s\le u}\{|H_F\star \widehat G_m(s) - H_\Pi\star \widehat G_m(s)|\}
 \end{equation*}
 then
\begin{align*}
\left|\psi_{H_F\star \widehat G_m}(u)-\psi_{H_\Pi\star \widehat G_m}(u)\right|
  & \le    \frac{\eta(1-\rho) \rho}{(1-\rho H_F(u/\delta) (1-\widehat G_m(\delta)))
   (1-\rho H_\Pi(u/\delta)(1-\widehat G_m(\delta)))}. \\
\end{align*}
\end{theorem}
 
An upper bound for $\sup_{0\le s\le u}|H_F\star\widehat{G}_m(s)-H_\Pi\star\widehat{G}_m(s)|$,
is suggested in the next Proposition.

\begin{Proposition}
 For $\delta>1$ we have that
\begin{align*}
   \sup_{0\le s\le u}|H_F\star \widehat G_m(s) - H_{\Pi}\star \widehat G_m(s)| 
    \le \eta(\delta),  
   \end{align*}
where 
\begin{align*}
 \eta(\delta):=\sup_{u/\delta\le s<\infty} \left|\overline H_\Pi(s)-\overline H_F(s)\right|\widehat G_m(\delta)
       +\sup_{0<s\le u/\delta}\left|H_F(s)-H_\Pi(s)\right|\left(1-\widehat G_m(\delta)\right).
\end{align*}
\end{Proposition}
The construction of the previous bounds depends on the availability of the 
distance between moment distributions $|H_F-H_\Pi|$, but the later might not always be available.
For such a case  we suggest a bound for such a quantity in Lemma \ref{Lemma.Dis.Moment}
for a specific type of approximating distributions $\Pi$.
The bound presented in there depends
on the cdf of the distribution $H_F$, the restricted expected value of the claim size distribution $F$ and its
approximation $\Pi$.

\section{Bounds for the numerical error of approximation}
\label{sec: Numerical error bound}

The probability of ruin of a reserve process as given in Theorems
\ref{Theorem.Ruin1} and \ref{Theorem.Ruin2} 
is not computable in exact form since the expression is given in terms of
various infinite series.  In practice, we can compute enough terms and then
truncate the series at a level where the error of truncation is smaller than
some desired precision. Since all terms involved are positive, such an 
approximation will provide an underestimate of the real ruin probability. 
In this section we compute error bounds for the approximation of the ruin probabilities 
occurred by truncating those series.

A close inspection of Theorems \ref{Theorem.Ruin1} and \ref{Theorem.Ruin2} reveals
that there will exist two sources of error due to truncation. The ruin probability can be seen
as the expected value of $\kappa_N$ where $N\sim\mbox{Poisson}(\xi u/s_1)$, so the 
first error  of truncation is $\Exp[\kappa_N|N\ge N_1]$,  we call $N_1$ the level of truncation for the ruin probability. Since the values of $\kappa_n$ are bounded
above by $1$, then it is possible to bound this error term with $\Prob(N\ge N_1)$ and
use Chernoff's bound \citep[cf. Theorem 9.3][]{Billingsley1995} to obtain an explicit expression
\begin{equation}\label{Chernoff}
 1-\zeta(N_1;\lambda)=\Prob(N>N_1)\le  \dfrac{\e^{-\lambda}(\e\cdot \lambda)^{N_1+1}}{(N_1+1)^{N_1+1}}.
\end{equation}
The second source of numerical error comes from truncating the infinite series 
induced by the scaling distribution $\Pi$; that is, we need to truncate
the series defining the terms $\mathcal{B}_i$, $\mathcal{C}_i$ and $\mathcal{D}_i$.
The following Lemma shows these truncated series  can be bounded by quantities
depending on the tail probability of $\Pi$ and the level of truncation $s_{N_2}$, where $N_2$ is the level of truncation for the scaling.

\begin{Lemma}\label{Lemma.Numerical}
Let $S\sim\Pi$ and define $\varepsilon_1=\Prob(S>s_{N_2})$ and $\varepsilon_2=\Exp[S;S>s_{N_2}]$. Then
\begin{align*} 
\mathcal B_i-\widetilde{\mathcal B}_i&\le\varepsilon_1,  & 0&\le i \, ,\\
\mathcal C_n-\widetilde{\mathcal C}_n&\le\varepsilon_1,  & \xi&\le n \, ,\\
\mathcal D_n-\widetilde{\mathcal D}_n&\le\varepsilon_2,  & \xi&\le n \, ,
\end{align*}
where $\widetilde{\mathcal{B}}_i$, $\widetilde{\mathcal{C}}_i$ and $\widetilde{\mathcal{D}}_i$
denote to the truncated series at $N_2$ terms.

\end{Lemma}

\begin{proof}
 If $0\le i< \xi-1$ then $\mathcal B_i=\widetilde{\mathcal B}_i=0$, otherwise if $\xi-1\le i \le N_2$ then
 \begin{align*}
   \mathcal B_i-\widetilde{\mathcal B}_i=
   \frac{\xi}{i+1}\sum\limits_{j=N_2+1}^{\infty} {\pi_j}\mbox{{bin}}(\xi;i+1,s_1/s_j)
   &\le   \sum\limits_{j=N_2+1}^{\infty} \pi_j=\varepsilon_1.
 \end{align*}
 Similarly, if $n\ge \xi$ then
 \begin{align*}
 \mathcal C_n-\widetilde{\mathcal C}_n=
   \sum\limits_{j=N_2+1}^{\infty}\pi_j 
  \mbox{Bin}(\xi-1;n,s_1/s_j)
   & \le\sum\limits_{j=N_2+1}^{\infty}  \pi_j=\varepsilon_1,
 \end{align*}
 while
  \begin{align*}
 \mathcal D_n-\widetilde{\mathcal D}_n
  &=\sum\limits_{j=N_2+1}^{\infty}\pi_j s_j\sum\limits_{k=0}^{\xi-1} \frac{\xi-k}{\xi}\mbox{bin}(k;n,s_1/s_j)
  \le\sum\limits_{j=N_2+1}^{\infty}\pi_js_j=\varepsilon_2.
 \end{align*}
\end{proof}

\subsection{Truncation error for $\psi_{\Pi\star G_m}$}

We start by writing the expression for the ruin probability in Theorem \ref{Theorem.Ruin1} (approximation A)
as a truncated series

\begin{equation*}
\widetilde\psi_{\Pi\star G_m}(u)=\e^{-\xi u/s_1}\sum_{n=0}^{N_1}
  \widetilde\kappa_n\dfrac{(\xi u)^n}{s_1^nn!},
\end{equation*}
where
\begin{equation*}
\label{eq.Trun1}
\widetilde\kappa_n=\begin{cases}
 \gamma\mu_F,   &   0\le n \le \xi-1,\\[.5cm]
     \gamma\mu_F  \left[\sum\limits_{i=\xi-1}^{n-1}\widetilde\kappa_{n-1-i}
            \widetilde{\mathcal B}_{i}	+ \widetilde{\mathcal C}_n\right],&\xi \le n \le N_1,
\end{cases}
\end{equation*}
with
\begin{align*}
\widetilde{\mathcal B}_i & = \frac{\xi}{i+1}
  \sum\limits_{j=1}^{N_2}{\pi_j}\mbox{{bin}}(\xi;i+1,s_1/s_j),
 &
\widetilde{\mathcal C}_n & = 
   \sum\limits_{j=1}^{N_2}\pi_j \mbox{{Bin}}(\xi-1;n,s_1/s_j).
\end{align*}

\begin{theorem}\label{Numerical.thm1}
Let $\varepsilon_1=\Prob(S>s_{N_2})$.  Then
\begin{equation*}
\psi_{\Pi\star G_m}(u)-\widetilde\psi_{\Pi\star G_m}(u)\le 
      \varepsilon_1 \left[\frac{\gamma\mu_F}{1-\gamma\mu_F} \left(\frac{\xi u}{s_1}\right) + \frac{2}{(1-\gamma\mu_F)^2}\e^{-\frac{(1-\gamma\mu_F)\xi u}{s_1}} \right]
      +\left(1-\zeta(N_1;\xi u/s_1)\right),
\end{equation*}
where $\zeta(N_1;\xi u/s_1)$ denotes
the cdf of a Poisson with parameter $\xi u/s_1$ and evaluated
at $N_1$.
 
\end{theorem}

\begin{proof}
 Observe that
 \begin{equation}\label{Numerical.Bound1}
    \psi_{\Pi\star G_m}(u)-\widetilde\psi_{\Pi\star G_m}(u)=
    \e^{-\xi u/s_1}\sum_{n=0}^{N_1}
      (\kappa_n-\widetilde\kappa_n)\dfrac{(\xi u)^n}{s_1^nn!}
      +\e^{-\xi u/s_1}\sum_{n=N_1+1}^{\infty}
      \kappa_n\dfrac{(\xi u)^n}{s_1^nn!}.
 \end{equation}
Firstly we consider the second term in the right hand side of \eqref{Numerical.Bound1}.
Using that $\kappa_n\le1$ we obtain that if $N_1>{\xi u}/{s_1}-1$, then
\begin{align*}
 \e^{-\xi u/s_1}\sum_{n=N_1+1}^{\infty} \kappa_n\dfrac{(\xi u)^n}{s_1^nn!}
   &\le\sum_{n=N_1+1}^{\infty}\e^{-\xi u/s_1}\dfrac{(\xi u)^n}{s_1^nn!}
   =\left(1-\zeta(N_1;\xi u/s_1)\right).
\end{align*}
Next we look into the first term of equation \eqref{Numerical.Bound1} and observe that
 \begin{equation*}\label{kappa.error}
\kappa_n-\widetilde\kappa_n=\begin{cases}
 0,   &   0 \le n \le \xi-1,\\[.5cm]
      \gamma\mu_F \left[\sum\limits_{i=\xi-1}^{n-1} \left(\kappa_{n-1-i}
            \mathcal B_{i}-\widetilde\kappa_{n-1-i}\widetilde{\mathcal B}_{i} \right)	+ \mathcal C_n-\widetilde{\mathcal{C}}_n\right],&\xi\le n\le N_1.
\end{cases}
\end{equation*}
Notice that if $n\ge \xi$ we can rewrite
\begin{align*}
  \sum\limits_{i=\xi-1}^{n-1} \left(\kappa_{n-1-i}
            \mathcal B_{i}-\widetilde\kappa_{n-1-i}\widetilde{\mathcal B}_{i} \right)
          =&  \sum\limits_{i=\xi-1}^{n-1} \left( (\kappa_{n-1-i}-\widetilde\kappa_{n-1-i})\mathcal B_i
     + \widetilde \kappa_{n-1-i}(\mathcal B_i-\widetilde{\mathcal{B}}_i) \right).
\end{align*}
Since $0<\widetilde\kappa_i\le\kappa_i\le 1$
for $0\le i$ then we can use the first part of Lemma
\ref{Lemma.Numerical} to obtain the following bound of the expression above
\begin{equation}\label{Big.Kappas}
  \sum\limits_{i=\xi-1}^{n-1}
     (\kappa_{i}-\widetilde\kappa_{i})\mathcal{B}_i+(n-\xi+1)\varepsilon_1.
\end{equation}
Putting 
\eqref{Big.Kappas} and the second part of Lemma
\ref{Lemma.Numerical} together we arrive at
\begin{align*}
\kappa_{n}-\widetilde\kappa_{n}\le \gamma\mu_F\left[\sum_{i=\xi-1}^{n-1}(\kappa_i-\widetilde\kappa_i)\mathcal{B}_i+(n-\xi+1)\varepsilon_1+\varepsilon_1
\right] 
&\le  \gamma\mu_F\left[\sup_{\xi-1\le i< n-1}(\kappa_i-\widetilde\kappa_i)\sum_{i=\xi}^{\infty}\mathcal{B}_i+(n-\xi+2)\varepsilon_1\right]\\
&\le  \gamma\mu_F\left[\sup_{\xi-1\le i< n-1}(\kappa_{n-1}-\widetilde\kappa_{n-1})+(n-\xi+2)\varepsilon_1\right].\\
\end{align*}
Note that $\sum_{i=\xi}^\infty \mathcal{B}_i=1$ follows from relating the formula of $\mathcal{B}_i$ 
to the probability mass function of a negative binomial distribution $\mbox{NP}(\xi,1-s_1/s_j)$.
Using the hypothesis that $\gamma\mu_F<1$ and induction it is not difficult to prove that
\begin{align*}
\kappa_n-\widetilde\kappa_n &\le \varepsilon_1 \sum_{i=2}^{n-\xi+2} i (\gamma\mu_F)^{n-\xi+3-i}
\le \varepsilon_1 \left[\frac{\gamma\mu_F}{1-\gamma\mu_F}n+\frac{2}{(1-\gamma\mu_F)^2}(\gamma\mu_F)^n \right].
\end{align*}
Inserting the bound above into the first term of equation \eqref{Numerical.Bound1}
and assuming that $\xi u/s_1>1$ we arrive at 
\begin{align*}
&\e^{-\xi u/s_1}\sum_{n=0}^{N_1}
      (\kappa_n-\widetilde\kappa_n)\dfrac{(\xi u)^n}{s_1^nn!} \le 
      \varepsilon_1 \left[\frac{\gamma\mu_F}{1-\gamma\mu_F} \left(\frac{\xi u}{s_1}\right) + \frac{2}{(1-\gamma\mu_F)^2}\e^{-\frac{(1-\gamma\mu_F)\xi u}{s_1}} \right] .
\end{align*}

%
\end{proof}

\begin{Remark}
 The term $\left(1-\zeta(N_1;\xi u/s_1)\right)$ can be bounded using Chernoff's bound
\begin{align*}
1-\zeta(N_1;\xi u/s_1)
   \le\dfrac{\e^{-\xi u/s_1}(\e^1\cdot \xi\cdot u/s_1)^{N_1+1}}{(N_1+1)^{N_1+1}}.
\end{align*}

\end{Remark}
\subsection{Truncation error for $\psi_{H_\Pi\star \widehat{G}_m}$}

{We write the ruin probability in Theorem \ref{Theorem.Ruin2} (approximation B) as a truncated series: 
\begin{equation*}
\widetilde\psi_{H_\Pi\star \widehat G_m}(u)=\sum_{n=0}^{N_1}\widetilde\kappa_n\dfrac{(\xi u/s_1)^n\e^{-\xi u/s_1}}{n!},
\end{equation*}
where
\begin{equation*}
\widetilde\kappa_n=\begin{cases}
 \gamma{\mu_F},   &   n=0,\\[.5cm]
    (\gamma{\mu_F}-1)\left(1+\dfrac{\gamma {\mu_F} s_1 }{{  {\mu_{\Pi}}} \xi}\right)^n+1,  &   1 \leq n \leq \xi,\\[.5cm]
     \dfrac{ \gamma {\mu_F} s_1}{{  {\mu_{\Pi}}} \xi }\sum\limits_{i=0}^{n-1}\widetilde\kappa_{n-1-i}
            \widetilde{\mathcal{C}}_{i}	+  \dfrac{ \gamma {\mu_F} }{{\mu_{\Pi}}}
            \widetilde{\mathcal{D}}_{n},&\xi< n.
\end{cases}
\end{equation*}
with
\begin{align*}
 \widetilde{\mathcal C}_{i}&=\sum\limits_{j=1}^{N_2}\pi_j\mbox{{Bin}}(\xi-1;i,s_1/s_j),
 &\widetilde{\mathcal D}_{n}&=\sum\limits_{j=1}^{N_2}\ \pi_j s_j\sum\limits_{k=0}^{\xi-1}\dfrac{\xi-k}{\xi}\mbox{{bin}}(k;n,s_1/s_j).
\end{align*}
}

The result and its proof are similar to the previous case.
 
\begin{theorem}
\label{Numerical thm2}
Let $S\sim\Pi$ and define  $\varepsilon_2=\Exp[S;S>s_{N_2}]$. Then
\begin{equation*}
\psi_{H_\Pi\star \widehat G_m}(u)-\widetilde\psi_{H_\Pi\star \widehat G_m}(u)
\le\varepsilon_2\e^{\frac{\xi u\gamma\mu_F}{s_1\mu_{\Pi}}}\left(\dfrac{\gamma\mu_F}{\mu_{\Pi}}+1\right)^{-\xi}
+\zeta(N_1;\xi u/s_1).
\end{equation*}

\end{theorem}
Building a bound for the numerical error of approximation B is more involved than
for approximation A given in Theorem \ref{Numerical.Bound1}.  The reason is that it is not simple to provide
a tight bound for the $\sum_{i=1}^\infty\mathcal{C}_i$ as for $\sum_{i=1}^\infty\mathcal{B}_i$. 
Notice that the bound is not as tight as in the case of Theorem \ref{Numerical.Bound1} and may not be of much practical use.
This aspect highlights an additional advantage of our first estimator.  

\begin{proof}
 Observe that
 \begin{equation}\label{kappa.error2}
\kappa_n-\widetilde\kappa_n=\begin{cases}
 0,   &   0\le n\le \xi,\\[.5cm]
     \dfrac{ \gamma\mu_F }{\mu_{\Pi}}\left[\dfrac{s_1}{\xi}\sum\limits_{i=\xi+1}^{n-1} \left(\kappa_{n-1-i}
           \mathcal C_i-\widetilde\kappa_{n-1-i}\widetilde{\mathcal C}_i \right)
            +\mathcal D_n-\widetilde{\mathcal D}_n\right],&\xi< n\le N_1.
\end{cases}
\end{equation}
The summation in  \eqref{kappa.error2} can be rewritten as
\begin{align*}
  \sum\limits_{i=\xi+1}^{n-1} \left(\kappa_{n-1-i}
           \mathcal C_i-\widetilde\kappa_{n-1-i}\widetilde{\mathcal C}_i \right)
            &=  \sum\limits_{i=\xi+1}^{n-1} \left( \left(\kappa_{n-1-i}-\widetilde\kappa_{n-1-i}\right)\mathcal C_i
     + \widetilde \kappa_{n-1-i}(\mathcal C_i-\widetilde{\mathcal C}_i) \right).
\end{align*}
Since $0<\widetilde\kappa_i\le\kappa_i\le 1$
and $\mathcal C_i\le 1$  for $0\le i$ then we can use the second part of Lemma
\ref{Lemma.Numerical} to obtain the following bound of the expression above
\begin{equation}\label{Big.Kappas2}
  \sum\limits_{i=\xi+1}^{n-1}
     \left(\kappa_{i}-\widetilde\kappa_{i}\right)+(n-\xi-1)\varepsilon_2.
\end{equation}
Putting  \eqref{Big.Kappas2} and the third part of Lemma
\ref{Lemma.Numerical} together we arrive at
\begin{align*}
   \kappa_{n}-\widetilde\kappa_{n}
   &\le\dfrac{\gamma\mu_F}{\mu_{\Pi}}\left(\frac{s_1}{\xi}\sum\limits_{i=\xi+1}\limits^{n-1}(\kappa_i-\widetilde \kappa_i)+(n-\xi)\varepsilon_2\right)
   \le\dfrac{\gamma\mu_F}{\mu_{\Pi}}\left(\sum\limits_{i=\xi+1}\limits^{n-1}(\kappa_i-\widetilde \kappa_i)+(n-\xi)\varepsilon_2\right).
\end{align*}
Induction yields that
\begin{align*}
\dfrac{\gamma\mu_F}{\mu_{\Pi}}\left(\sum\limits_{i=\xi+1}\limits^{n-1}(\kappa_i-\widetilde \kappa_i)+(n-\xi)\varepsilon_2\right)
 &= \varepsilon_2\left(\left(\dfrac{\gamma\mu_F}{\mu_{\Pi}}+1\right)^{n-\xi}-1\right)\\
 &\le \varepsilon_2\left(\dfrac{\gamma\mu_F}{\mu_{\Pi}}+1\right)^{n-\xi}.
\end{align*}
Hence we arrive at 
\begin{equation*}
\e^{-\xi u/s_1}\sum_{n=0}^{N_1}
      (\kappa_n-\widetilde\kappa_n)\dfrac{(\xi u)^n}{s_1^nn!} \le \varepsilon_2\e^{\frac{\xi u\gamma\mu_F}{s_1\mu_{\Pi}}}\left(\dfrac{\gamma\mu_F}{\mu_{\Pi}}+1\right)^{-\xi}.
\end{equation*}
 
\end{proof}

\section{Numerical implementations}
\label{sub.Implementation}
We briefly discuss some relevant aspects of the implementation
of Theorems \ref{Theorem.Ruin1} and \ref{Theorem.Ruin2}.

Suppose we want to approximate a distribution $\widehat F$ via Erlangized scale mixtures. 
The selection of the parameter $\xi\in\nat$ of the Erlang distribution 
boils down to selecting a value $\xi$ large enough so the bound provided in
Theorem \ref{Erlagization.Bound.A} and Theorem \ref{First error bound} is smaller than 
a preselected precision.  It is however not recommended to select a value which is too large since
this will require truncating at higher levels and thus resulting in a much slower algorithm 
(this will be further discussed below). 

The most critical aspect for an efficient implementation is the selection of 
an appropriate approximating distribution $\Pi$.
The selection can be made rather arbitrary but we suggest 
the following general family of discrete distributions:
\begin{definition}\label{Def: step distribution}
Let $W=\{w_i:i\in\mathbb{Z}^+\}$ and $\Omega_\Pi=\{s_i:i\in\mathbb{Z}^+\}$ be sets of strictly increasing nonnegative values such that
$w_0=s_0=\inf\{s:F(s)>0\}$ and for all $k\in\mathbb{N}$ it holds that
\begin{equation*}
 s_{k}\le w_k\le s_{k+1}.
\end{equation*}
Then we define the distribution $\Pi$ as 
\begin{equation*}
 \Pi(s):=\sum_{k=0}^\infty F(w_{k})\Ind_{[s_{k},s_{k+1})}(s).
\end{equation*} 
\end{definition}

The distribution $\Pi$ is a discretized approximating distribution which is upcrossed by $F$ in every interval $(w_{k-1},w_k)$.
This type of approximation is rather general as we can consider general approximations by selecting $s_k\in(w_{k-1},w_k)$,
approximations from \emph{below} by setting $s_k=w_{k}$,
approximations from \emph{above} by setting $s_k=w_{k-1}$, or the middle point (see Figure~\ref{Fig.Discretization}). 
Heuristically, one might expect to reduce the error of approximation by selecting the middle point
(this was our selection in our experimentations below).
\begin{figure}[h]
 \caption{Four alternative discretizations of a Pareto(2) distribution. Panel (a) shows a general approximation.
 Panels (b) and (c) shows approximations from above and below respectively, while in the last panel (d)
 we have selected the middle point.}
 \label{Fig.Discretization}
 \includegraphics[scale=0.45]{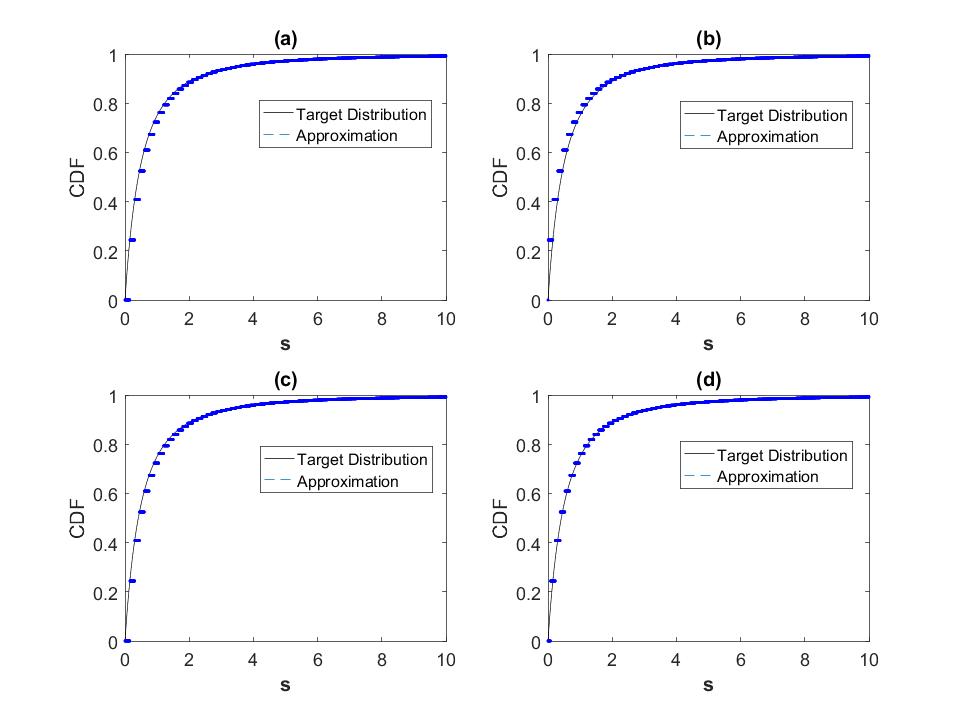}
\end{figure}

%

In practice we can just  compute a finite number of terms $\pi_k$, so we end up with an improper distribution. 
This represents a serious issue because truncating at lower levels affects the quality of the
approximation in the tail regions.  Computing a larger number of terms is not often
an efficient alternative since the computational times become rapidly unfeasible. 
Thus, our ultimate goal will be to select among the partitions of certain fixed
size (we restrict the partition size since we assume we have a limited computational budget), 
the one that \emph{minimizes} the distance $|F-\Pi|$, in particular in the tails. 
In our numerical experimentations we found that an arithmetic progression required
a prohibitively large number of terms to obtain sharp approximations in the tail.  We obtained
better results using geometric progressions as these can provide better approximations
with a reduced number of terms. Moreover, since the sequence determining the probability mass 
function converges faster to $0$, then it is easier to compute enough terms so 
for practical purposes it is equivalent to work with a proper distribution.

The speed of the algorithm is heavily determined by the total number of 
terms  of the infinite series in Theorems \ref{Theorem.Ruin1} and
\ref{Theorem.Ruin2} computed.   
Since the probability of interest can be seen as the expected
value $\Exp[\kappa_N]$ where $N\sim\mbox{Poisson}(\xi u/s_1)$, 
it is straightforward to see that the total
number $N_1$ of terms needed to provide an accurate approximation 
is directly related to the value
$\xi u/s_1$. 
Thus, large values of $\xi$ and $u$ combined with small values of $s_1$ 
will require longer computational times.  
Since smaller values of $\xi$ and 
larger values of $s_1$ will typically result in increased errors of approximation, there
will be a natural trade-off between speed and precision in the selection of these values.
In our numerical experiments below we have 	selected empirically these values
with the help of the error bounds found in the previous sections.

It is also worth noting that  the calculation of the value $\kappa_n$ for $n>\xi$ in both Theorems 
\ref{Theorem.Ruin1} and \ref{Theorem.Ruin2} requires the evaluation of 
the binomial probability mass functions
$\mbox{bin}(\cdot;i,\cdot)$ for all $i=\xi,\dots,N_1$.
While the computation of such probabilities is relatively simple, it is not particularly
efficient to compute each term separately because the computational times
become very slow as $n$ goes to infinity. Due to the recursive nature of the
coefficients $\kappa_n$ one may incur in significant numerical errors if the 
the binomial probabilities are not calculated at a high precision. For more details, see for instance
\cite{Loader2000} for recommended strategies that can be used to increase the speed
and accuracy of the binomial probabilities.    

Finally, we remark that 
the speed of the implementation can be significantly improved by using parallel computing.
In our implementations below we have broken the series into smaller pieces and we have sent this to 
an HPC (high performance computing) facility to run independent units of work.

\subsection{Numerical examples}

In this section, we show the accuracy of our approximation A through the following 
Pareto example. In such an example, the claim sizes are Pareto distributed, so their 
integrated tails are regularly varying. 
The exact values of the ruin probability are given in \cite{Ramsay2003solution}, 
and are now considered a classical benchmark for comparison purposes.  We have limited our
numerical experiments to the Pareto with parameter 2 and net profit condition close to 0  ($\rho\to1$) as this 
is one of the most challenging ruin probabilities we could find for which there are results
available for comparison.
\begin{Example}[Pareto claim sizes]
We consider a Cram\'er--Lundberg model with unit premium rate, and claim sizes distributed  
according to a Pareto distribution with a single parameter $\phi>1$ with support on the positive real axis, mean $1$ 
and having the following cumulative distribution function
\begin{equation}\label{Pareto}
F(x) = 1 - \left(1+\frac{x}{\phi-1}\right)^{-\phi}, \; \text{for} \; x > 0 \; \text{and} \; \phi >1,
\end{equation}
(other parametrizations of the Pareto distribution are common as well).
The integrated tail of the above distribution is regularly varying with parameter $\phi-1$:
\begin{align*}
\widehat{F}(x) &= \frac{1}{\mu_{F}} \int_0^x \overline{F}(t) \,\text{d} t 
       = \int_0^x \left(1+\frac{t}{\phi-1}\right)^{-\phi} \,\text{d} t 
       = 1-\left(1+\frac{x}{\phi-1}\right)^{-(\phi-1)}.
\end{align*}
The parameters of the risk model selected were $\rho=0.95$, $\phi=2$.
We implemented the approximation A in Theorem~\ref{Theorem.Ruin1} and its analysis is presented next.
For comparisons purposes we also included the approximation B in Theorem~\ref{Theorem.Ruin2}, but overall
we found that it is less accurate, much slower and more difficult to analyze since its bounds are not tight enough.

First we analyzed the Erlangization error for approximation A.  For this example, it is possible to compute the bound given by 
Theorem \ref{Th.Vatamidou} for values of $\xi=100,500,1000$. The bound appears to be tighter
for smaller values of $\rho$ while it gets loosen as long as the value of $\rho\to1$.  The bound also 
increases as $u\to\infty$, so probabilities of ruin with large initial reserves will be more difficult
to approximate.  The bound appears to decrease proportionally in $\xi$ but in practice, we didn't noticed
significant changes in the numerical approximation of the probability of ruin for values of $\xi$ larger than $100$.
Nevertheless, since larger values of $\xi$ affect the speed of the algorithm 
we settled with a value of $\xi=100$ which already gave good results overall.
\begin{table}[htb]
	\centering
	\caption{Erlagization Error Bounds.}\label{ErrorBound1}
	\begin{tabular} {p{1.5cm}|p{2.5cm}|p{2.5cm}|p{2.5cm}|p{2.5cm}}  
	\hline  \hline 	
	$u$ &$\xi$=100  & $\xi$=500 & $\xi$=1000 \\
	\hline  \hline 	
	1     &  $2.5736 \times 10^{-4}$ & $5.1724 \times 10^{-5}$ & $2.5856 \times 10^{-5}$  \\
	5     &  $1.6324 \times 10^{-3}$ & $3.2839 \times 10^{-4}$ & $1.6418 \times 10^{-4}$  \\
	10    &  $3.8081 \times 10^{-3}$ & $7.6641 \times 10^{-4}$ & $3.8319 \times 10^{-4}$  \\
	30    &  $1.0884 \times 10^{-2}$ & $2.1911 \times 10^{-3}$ & $1.5128 \times 10^{-3}$  \\
	50    &  $1.5028 \times 10^{-2}$ & $3.0257 \times 10^{-3}$ & $2.6455 \times 10^{-3}$  \\
	100   &  $2.0055 \times 10^{-2}$ & $4.0379 \times 10^{-3}$ & $2.0190 \times 10^{-3}$  \\
	500   &  $2.6279 \times 10^{-2}$ & $5.2911 \times 10^{-3}$ & $2.6455 \times 10^{-3}$  \\
	1000  &  $2.7265 \times 10^{-2}$ & $5.4896 \times 10^{-3}$ & $2.7448 \times 10^{-3}$  \\
	\hline \hline 	
	\end{tabular}
\end{table}
	
Next we constructed a discrete approximating distribution $\Pi$ by considering a discretized Pareto  
supported over the geometric progression $\{\e^{t_0}, \e^{t_1}, \e^{t_2}, \cdots\}$, where $t_k=t_0+k/K$.
It is rather clear that a finer partition of the interval
$[0,\infty)$ would yield a better approximation and this would be attained by letting the value of $t_0\to-\infty$
and $K\to\infty$.  However, small values of $s_1:=\e^{t_0}$ affect severely the speed of the algorithm (see the discussion
above) while in practice not much precision is gained by taking it too close to $0$.  A similar
trade-off in speed and precision occurs by letting $K\to\infty$. For this example we have selected these values empirically with 
the help of the bound in Theorem \ref{Theorem.error.discret.I} and Proposition \ref{Prop.error.discrete.I}.
We settled with $t_0=-3$ and $K=270$ for all the examples.
The results are in the first column in the Table \ref{ErrorBound2} below.

\begin{table}[htb]
	\centering
	\caption{Error Bounds}\label{ErrorBound2}
	\begin{tabular} {p{1.5cm}|p{3.5cm}|p{3.5cm}}  
	\hline  \hline 	
	$u$ & Discretization Error  &  Truncation  Error $N_2$   \Tstrut\Bstrut\\
	\hline  \hline 	
	1     &  $6.2002 \times 10^{-4}$ &  $3.6522 \times 10^{-9}$\\
	5     &  $6.3557 \times 10^{-5}$ &  $1.8261 \times 10^{-8}$\\
	10    &  $3.1590 \times 10^{-5}$ &  $3.6522 \times 10^{-8}$\\
	30    &  $1.0617 \times 10^{-5}$ &  $1.0957 \times 10^{-7}$\\
	50    &  $6.4216 \times 10^{-6}$ &  $1.8261 \times 10^{-7}$\\
	100   &  $3.2491 \times 10^{-6}$ &  $3.6522 \times 10^{-7}$\\
	500   &  $6.6918 \times 10^{-7}$ &  $1.8261 \times 10^{-6}$\\
	1000  &  $3.3891 \times 10^{-8}$ &  $3.6522 \times 10^{-6}$\\
	\hline \hline 	
	\end{tabular}
\end{table}

Next we selected the truncation levels. In the case of $N_1$ we were able to select a natural number large enough 
such that the truncation error was smaller than the floating point precision without increasing 
significantly the computational times. 
This selection implies that the third term in the bound for the truncation error given in Theorem \ref{Numerical.thm1}
is eliminated for practical purposes.
 As for $N_2$, we choose the smallest integer $N_2$ such that $\varepsilon_1 < 9.5701 \times 10^{-14}$. 
The error bounds are presented in the {last} column of Table \ref{ErrorBound2}. Notice
that the dominant term in Theorem \ref{Numerical.thm1} is asymptotically linear in $u$.  This pattern
is also observed numerically as the error bound appears increasing linearly with respect to $u$, thus
providing empirical evidence that suggests this bound is tight.

The numerical results for the probabilities of ruin are now summarized in Table~\ref{ParExp}. 
The results show that the approximated ruin probabilities are remarkably close to the true value calculated using equation (20) of~\cite{Ramsay2003solution}.  

\begin{table}[htb]
	\centering
	\caption{Approximation of ruin probabilities when claim sizes are Pareto distributed, $\rho=0.95$ and $\phi=2$.} \label{ParExp}

	\begin{tabular} {p{2cm}||p{3cm}|p{3cm}||p{2cm}}  
		\hline 		\hline 			
  \multirow{2}{*}{$u$} & \multicolumn{2}{c||}{Approximation} & \multirow{2}{*}{Ramsay} \Tstrut\Bstrut\\
		\cline{2-3} \Tstrut\Bstrut
	&	Theorem 3.4 & Theorem 3.5 &  \\
		\hline \hline
		1   &  0.915506746  & 0.915513511
  &   0.915525781
    \Tstrut\Bstrut\\
		5  &  0.837217038 & 0.837576604
  &  0.837251342
     \Tstrut\Bstrut\\ 
		10  &  0.770595774 &  0.771230756
  &  0.770605760
     \Tstrut\Bstrut\\
		30  & 0.599128897 & 0.600357750
   &   0.599042454
    \Tstrut\Bstrut\\
		50  & 0.489803156 & 0.491286606
   &  0.489654166
    \Tstrut\Bstrut\\
		100 & 0.325521064  &0.327119739
  &   0.325305086
     \Tstrut\Bstrut\\
		500 & 0.059229343 & 0.059800534
  &   0.059131409
     \Tstrut\Bstrut\\
		1000 & 0.024594577 & 0.024819606
  & 0.024544601
       \Tstrut\Bstrut\\
		\hline  \hline
	\end{tabular}
\end{table}
\end{Example}
The numerical results above were produced with the same values of $\xi$, $t_0$ and $K$.  We remark
that as long as the value of $u$ increases, then the numerical approximation appears to be less
sharp, but this can be improved by increasing the value of $K$ (this would make the partition
finer) and to a lesser extent by reducing the value of $t_0$ (improving the approximation of the target distribution
in a vicinity of 0).  The approximation was less sensitive to increases in the value of $\xi$ but makes
it considerably slower.

\section{Conclusion}
\label{mysec7}
\cite{Bladt2015} remarked that the  
family of phase-type scale mixtures could be used to provide sharp approximations
of heavy-tailed claim size distributions. 
In our work, we addressed such a remark and provided a simple systematic methodology to approximate any nonnegative
continuous distribution within such a family of distributions.  
We employed the results of \cite{Bladt2015} and provided simplified
expressions for the probability of ruin in the classical 
Cram\'er--Lundberg risk model. In particular we opted to approximate
the integrated tail distribution $\widehat F$ rather than the claim sizes
as suggested in \cite{Bladt2015};  
we showed that such an alternative approach results in a more accurate 
and simplified approximation for the associated ruin probability. 
We further provided bounds for the error of approximation induced by approximating
the integrated tail distribution  as well as the error induced by the truncation 
of the infinite series. 
Finally, we illustrated the accuracy of our proposed method by computing the ruin probability
of a Cram\'er-Lundberg reserve process  where the claim sizes are heavy-tailed.
Such an example is classical but often considered challenging due to the heavy-tailed
nature of the claim size distributions and the value of the net profit condition.

\paragraph*{Acknowledgements}
The authors thank Mogens Bladt for multiple discussions on the ideas which originated this paper
and an anonymous referee who provided a detailed review that helped to improve the quality of this paper.
OP is supported by the CONACYT PhD scholarship No. 410763 sponsored by the Mexican Government.
LRN is supported by ARC grant DE130100819. 
WX is supported by IPRS/APA scholarship at The University of Queensland. HY is supported by APA scholarship at The University of Queensland.

\bibliographystyle{chicagoa}
\bibliography{InfiniteBib,StandardBib}

\section{Appendix:  Bounds for errors of approximation}

In the first subsection of this appendix we provide a refined bound for one of the approximations 
proposed in the main section. In the second subsection of the appendix we provide an auxiliary result that
will be useful for the numerical computation of one of the bounds proposed.

\subsection{Refinements for the Erlangization error of $\psi_{H_\Pi\star \widehat G_m}$}

Through Theorem \ref{Refinement.First.Bound} we provide a refinement of the bound proposed in Theorem \ref{First error bound}.
This refined bound is much tighter although more difficult to construct and implement. The following preliminary results are needed first.

\begin{Lemma}\label{epsilon.delta}
For any $\delta>0$, define 
\[\epsilon_m(\delta) := \int_0^\delta G_m(s)\dd s.\]
Then 
\[\epsilon_m(\delta) = e^{-\xi\delta}\left(\sum_{k=0}^{\xi-1}\frac{(\xi\delta)^k}{k!} - \sum_{k=0}^{\xi-2}\delta\frac{(\xi\delta)^k}{k!}\right).\]
\end{Lemma}
Notice that $\epsilon_m(1) = \epsilon_m=1-\widehat G_m(1)$ while $\epsilon_m(0)=0$.
\begin{proof}
Consider
\begin{align*}
\int_{0}^\delta e^{-\xi s}(\xi s)^n\dd s & = -\frac{e^{-\xi\delta}}{\xi}(\xi\delta)^n + n\int_0^\delta e^{-\xi s}(\xi s)^{n-1}\dd s\\
& = \frac{n!}{\xi} -\frac{e^{-\xi\delta}}{\xi}\left(\sum_{i=0}^n\frac{n!}{i!}(\xi\delta)^i \right),
\end{align*}
so that
\begin{align*}
\epsilon_m(\delta) & = \int_0^\delta\left(1 - \sum_{n=0}^{\xi-1}\frac{1}{n!}e^{-\xi s}(\xi s)^n\right)\dd s\\ 
& = 1 - \sum_{n=0}^{\xi -1}\frac{1}{n!}\int_0^\delta e^{-\xi s}(\xi s)^n\dd s\\
& = 1 - \sum_{n=0}^{\xi -1}\frac{1}{n!}\left(\frac{n!}{\xi} -\frac{e^{-\xi\delta}}{\xi}\left(\sum_{i=0}^n\frac{n!}{i!}(\xi\delta)^i \right)\right)\\
& = \frac{e^{-\xi\delta}}{\xi}\sum_{n=0}^{\xi-1}\sum_{k=0}^n \frac{(\xi\delta)^k}{k!}\\
& = \frac{e^{-\xi\delta}}{\xi}\sum_{k=0}^{\xi-1}(\xi-k) \frac{(\xi\delta)^k}{k!}\\
& = e^{-\xi\delta}\left(\sum_{k=0}^{\xi-1}\frac{(\xi\delta)^k}{k!} - \sum_{k=0}^{\xi-2}\delta\frac{(\xi\delta)^k}{k!}\right).
\end{align*}
\end{proof}

\begin{Lemma}\label{lem:delta12}
Let $0\le \delta_1\le 1$ and $1\le \delta_2\le\infty$.  Define 
\begin{equation*}
 A_{\delta_1}=[\delta_1,1]^n,\qquad A^{\delta_1,\delta_2}=[\delta_1,\delta_2]^n\setminus[\delta_1,1]^n.
\end{equation*}
Then
\begin{align*}
&\idotsint\limits_{A_{\delta_1}} \left| \prod_{i=1}^n {\Ind}_{[0,1)}(s_i) - \prod_{i=1}^n \widehat{g}_m(s_i)\right| \dd s_1 \dots \dd s_n\\
&\qquad\qquad = (1-\delta_1)^n - \left(1-\delta_1 - \epsilon_m(1) + \epsilon_m(\delta_1)\right)^n\\
&\idotsint\limits_{A^{\delta_1, \delta_2}} \left| \prod_{i=1}^n {\Ind}_{[0,1)}(s_i) - \prod_{i=1}^n \widehat{g}_m(s_i)\right| \dd s_1 \dots \dd s_n\\
&\qquad\qquad = \left(\widehat G_m(\delta_2)-\widehat G_m(\delta_1)\right)^n - \left(\widehat{G}_m(1)-\widehat{G}_m(\delta_1)\right)^n.
\end{align*}
\end{Lemma}

\begin{proof}
\begin{align*}
&\idotsint\limits_{A_{\delta_1}} \left| \prod_{i=1}^n {\Ind}_{[0,1)}(s_i) - \prod_{i=1}^n \widehat{g}_m(s_i)\right| \dd s_1 \dots \dd s_n\\
&\quad = \idotsint\limits_{A_{\delta_1}}1- \prod_{i=1}^n \widehat g_m(s_i)\dd s_1 \dots \dd s_n\\
& \quad =(1-\delta_1)^n - \left(\int_{\delta_1}^1 \widehat g_m(s)\dd s\right)^n\\
& \quad =(1-\delta_1)^n - \left(\int_{\delta_1}^1 1-G_m(s)\dd s\right)^n\\
& \quad =(1-\delta_1)^n - \left(1-\delta_1 - \epsilon_m + \epsilon_m(\delta_1)\right)^n.
\end{align*}

For the second equality, notice that
\begin{align*}
\idotsint\limits_{[\delta_1, \delta_2]^n}\prod_{i=1}^n \widehat{g}_m(s_i) \dd s_1 \dots \dd s_n
&= \left(\int_{\delta_1}^{\delta_2}\widehat{g}_m(s)\dd s\right)^n
= \left(\widehat G_m(\delta_2)-\widehat{G}_m(\delta_1)\right)^n,
\end{align*}

so that 
\begin{align*}
\idotsint\limits_{A^{\delta_1,\delta_2}} \left| \prod_{i=1}^n {\Ind}_{[0,1)}(s_i) - \prod_{i=1}^n \widehat{g}_m(s_i)\right| \dd s_1 \dots \dd s_n
&  = \idotsint\limits_{A^{\delta_1,\delta_2}}\prod_{i=1}^n \widehat{g}_m(s_i)\dd s_1 \dots \dd s_n\\
&= \idotsint\limits_{[\delta_1,\delta_2]^n\setminus [\delta_1,1]^n} \prod_{i=1}^n \widehat{g}_m(s_i) \dd s_1 \dots \dd s_n\\
& =  \left(\widehat G_m(\delta_2)-\widehat{G}_m(\delta_1)\right)^n - \left(\widehat G_m(1)-\widehat{G}_m(\delta_1)\right)^n.
\end{align*}
\end{proof}

\begin{Corollary}\label{cor:epsilon12}
Fix $\delta_2\in(1,\infty)$. Then there exists $\delta_1\in[0,1]$ such that
\begin{align*}
& \idotsint\limits_{A_{\delta_1}} \left| \prod_{i=1}^n {\Ind}_{[0,1)}(s_i) - \prod_{i=1}^n \widehat{g}_m(s_i)\right| \dd s_1 \dots \dd s_n \\
& \quad=\idotsint\limits_{A^{\delta_1,\delta_2}} \left| \prod_{i=1}^n {\Ind}_{[0,1)}(s_i) - \prod_{i=1}^n \widehat{g}_m(s_i)\right| \dd s_1 \dots \dd s_n,
\end{align*}
where $A_{\delta_1}= (\delta_1, 1)^n$ and $A^{\delta_1,\delta_2} = (\delta_1, \delta_2)^n\setminus (\delta_1, 1)^n$.
\end{Corollary}
\begin{proof}
Define the following functions with domain $[0,1]$:
\begin{align*}
p(\delta)& := \idotsint\limits_{A_{\delta}} \left| \prod_{i=1}^n {\Ind}_{[0,1)}(s_i) - \prod_{i=1}^n \widehat{g}_m(s_i)\right| \dd s_1 \dots \dd s_n, \quad\mbox{and}\\
q_{\delta_2}(\delta) &:=\idotsint\limits_{A^{\delta,\delta_2}} \left| \prod_{i=1}^n {\Ind}_{[0,1)}(s_i) - \prod_{i=1}^n \widehat{g}_m(s_i)\right| \dd s_1 \dots \dd s_n.
\end{align*}
By Lemma \ref{lem:delta12}, both functions are continuous, $p$ is non-increasing and $q_{\delta_2}$ is non-decreasing. The image of $q_{\delta_2}$ is contained in $[0,1-\epsilon_m(1)]$ while the image of $p$ is exactly $[0,1-\epsilon_m(1)]$. All these mean that there exists a point $\delta_1\in[0,1]$ such that $q_{\delta_2}(\delta_1) = p(\delta_1)$, concluding the proof.
\end{proof}

The following Corollary follows immediately from Lemma \ref{lem:delta12} by setting $\delta_1=0$ and $\delta_2=\infty$.  This Corollary is needed in the proof of Theorem \ref{First error bound}.
\begin{Corollary}\label{Cor.Int}
\begin{align*}
\idotsint\limits_{\RL^n} \left| \prod_{i=1}^n {\Ind}_{[0,1)}(s_i) - \prod_{i=1}^n \widehat{g}_m(s_i)\right| \dd s_1 \dots \dd s_n
  = 2(1-(1-\epsilon_m)^n).
\end{align*}
\end{Corollary}

The following provides a simple bound between the difference of the $n$-th convolution of any distribution function $F$ with density $f$ evaluated at two different points.

\begin{Lemma}\label{Lemma.Dist.Conv}
Let $\widehat F$ be any continuous distribution function supported on $[0,\infty)$ with density function $f$  and fix $b>a>0$. Then 
\[\vert \widehat F^{*n}(a)-\widehat F^{*n}(b)\vert \le \Delta_{a,b}^{\widehat F}\cdot \widehat F^{n-1}(b),\]
where $\Delta_{a,b}^{\widehat F} := \sup\left\{\widehat F(b-a+c) - \widehat F(c):c\in(0,a)\right\}$.
\end{Lemma}
\begin{proof}
\begin{align*}
\vert \widehat F^{*n}(a)-\widehat F^{*n}(b) \vert &= \int_{a}^b f^{*n}(u)\dd u
 = \int_a^b\int_0^u f(u-s)f^{*(n-1)}(s)\dd s\dd u\\
&\le \int_a^b\int_0^b f(u-s)f^{*(n-1)}(s)\dd s\dd u = \int_0^b f^{*(n-1)}(s) \int_a^b f(u-s)\dd u \dd s
\end{align*}
Then there exists a constant $c\in(0,a)$ such that the previous expression
is equal to
\begin{align*}
 \left\{\int_a^b f(u-c)\dd u\right\} \int_0^b f^{*(n-1)}(s) \dd s &= \left\{\int_{a-c}^{b-c} f(u)\dd u\right\} \widehat F^{*n-1}(b)
  \le \left\{\int_{c}^{(b-a)+c} f(u)\dd u\right\} \widehat F^{n-1}(b).
\end{align*}
The result follows from taking the supremum over $c$.
\end{proof}

\begin{Lemma}\label{Lem:areasmallFs}
Let $\{X_n'\}$ be a sequence of i.i.d. random variables with common distirbution $H_F$. Fix $\delta_2\in(1,\infty)$ and let $\delta_1\in(0,1)$ be as in Corollary \ref{cor:epsilon12}. Then 
\begin{align*}
& \sum_{n=0}^\infty(1-\rho)\rho^n \idotsint\limits_{A_{\delta_1}\cup A^{\delta_1,\delta_2}} \mathbb{P}(s_1 X_1' + \dots + s_n X_n'\le u ) \left| \prod_{i=1}^n {\Ind}_{[0,1)}(s_i)-\prod_{i=1}^n \widehat{g}_m(s_i)    \right| \dd s_1 \dots \dd s_n\\
& \qquad = (1-\rho)\rho \Delta_{u/\delta_2, u/\delta_1}^{{H_F}} \left(\frac{\epsilon_m(1) - \epsilon_m(\delta_1)}{(1-\rho F(u\delta_1))(1-\rho F(u\delta_1)(1-\epsilon_m(1) + \epsilon_m(\delta_1))}\right).
\end{align*}
\end{Lemma}
\begin{proof}
Clearly, for any $(r_1,\dots, r_n)\in A_{\delta_1}$ and $(s_1,\dots, s_n)\in A^{\delta_1, \delta_2}$, we have that
\begin{align*}
\mathbb{P}(r_1 X_1' + \dots + r_n X_n'\le u )& \le H_F^{*n}(u/\delta_1), \quad\mbox{and}\\
\mathbb{P}(s_1 X_1' + \dots + s_n X_n'\le u ) &\ge H_F^{*n}(u/\delta_2),
\end{align*}
so that
\begin{align*}
&\left|\idotsint\limits_{A_{\delta_1}\cup A^{\delta_1,\delta_2}} \mathbb{P}(s_1 X_1' + \dots + s_n X_n'\le u ) \left( \prod_{i=1}^n {\Ind}_{[0,1)}(s_i)-\prod_{i=1}^n \widehat{g}_m(s_i)    \right) \dd s_1 \dots \dd s_n\right|\\
&\qquad = \left|\idotsint\limits_{A_{\delta_1}} \mathbb{P}(s_1 X_1' + \dots + s_n X_n'\le u ) \left| \prod_{i=1}^n {\Ind}_{[0,1)}(s_i)-\prod_{i=1}^n \widehat{g}_m(s_i)    \right| \dd s_1 \dots \dd s_n\right.\\
&\qquad\qquad - \left.\idotsint\limits_{A^{\delta_1, \delta_2}} \mathbb{P}(s_1 X_1' + \dots + s_n X_n'\le u ) \left| \prod_{i=1}^n {\Ind}_{[0,1)}(s_i)-\prod_{i=1}^n \widehat{g}_m(s_i)    \right| \dd s_1 \dots \dd s_n\right|\\
&\qquad \le \left|H_F^{*n}(u/\delta_1)\idotsint\limits_{A_{\delta_1}} \left| \prod_{i=1}^n {\Ind}_{[0,1)}(s_i)-\prod_{i=1}^n \widehat{g}_m(s_i)    \right| \dd s_1 \dots \dd s_n\right.\\
&\qquad\qquad - \left.H_F^{*n}(u/\delta_2)\idotsint\limits_{A^{\delta_1, \delta_2}} \left| \prod_{i=1}^n {\Ind}_{[0,1)}(s_i)-\prod_{i=1}^n \widehat{g}_m(s_i)    \right| \dd s_1 \dots \dd s_n\right|\\
&\qquad = (H_F^{*n}(u/\delta_1) - H_F^{*n}(u/\delta_2))\idotsint\limits_{A_{\delta_1}} \left| \prod_{i=1}^n {\Ind}_{[0,1)}(s_i)-\prod_{i=1}^n \widehat{g}_m(s_i)    \right| \dd s_1 \dots \dd s_n\\
&\qquad = (H_F^{*n}(u/\delta_1) - H_F^{*n}(u/\delta_2))[(1-\delta_1)^n - \left(1-\delta_1 - \epsilon_m(1) + \epsilon_m(\delta_1)\right)^n].
\end{align*}
Using the previous results and Lemma \ref{Lemma.From.Hell} we get that
\begin{align*}
& \sum_{n=0}^\infty(1-\rho)\rho^n \idotsint\limits_{A_{\delta_1}\cup A^{\delta_1,\delta_2}} \mathbb{P}(s_1 X_1' + \dots + s_n X_n'\le u ) \left| \prod_{i=1}^n {\Ind}_{[0,1)}(s_i)-\prod_{i=1}^n \widehat{g}_m(s_i)    \right| \dd s_1 \dots \dd s_n\\
&\qquad \le \sum_{n=0}^\infty(1-\rho)\rho^n \{H_F^{*n}(u/\delta_1) - H_F^{*n}(u/\delta_2)\}[(1-\delta_1)^n - \left(1-\delta_1 - \epsilon_m(1) + \epsilon_m(\delta_1)\right)^n]\\
&\qquad \le \sum_{n=1}^\infty(1-\rho)\rho^n\{\Delta_{u/\delta_2, u/\delta_1}^{{H_F}} H_F^{*(n-1)}(u/\delta_1)\}[(1-\delta_1)^n - \left(1-\delta_1 - \epsilon_m(1) + \epsilon_m(\delta_1)\right)^n]\\
&\qquad \le \sum_{n=1}^\infty(1-\rho)\rho^n\{\Delta_{u/\delta_2, u/\delta_1}^{{H_F}} H_F^{(n-1)}(u/\delta_1)\}[(1-\delta_1)^n - \left(1-\delta_1 - \epsilon_m(1) + \epsilon_m(\delta_1)\right)^n]\\
& \qquad =(1-\rho)\rho\Delta_{u/\delta_2, u/\delta_1}^{{H_F}}  \sum_{n=0}^\infty \rho^n H_F^{n}(u/\delta_1)[(1-\delta_1)^{n+1} - \left(1-\delta_1 - \epsilon_m(1) + \epsilon_m(\delta_1)\right)^{n+1}]\\
& \qquad = (1-\rho)\rho\Delta_{u/\delta_2, u/\delta_1}^{{H_F}} \left( \frac{1-\delta_1}{1-\rho H_F(u/\delta_1)(1-\delta_1)} - \frac{1-\delta_1 - \epsilon_m(1) + \epsilon_m(\delta_1)}{1-\rho H_F(u/\delta_1)(1-\delta_1 - \epsilon_m(1) + \epsilon_m(\delta_1))}\right).
\end{align*}
\end{proof}

\begin{theorem}\label{Refinement.First.Bound}
For any fixed $\delta_2\in(1,\infty)$ let $\delta_1$ be as in Corollary \ref{cor:epsilon12}. Then
\begin{align*}
 \left| \psi_{H_F\star U}(u) - \psi_{H_F\star \widehat G_m}(u) \right|
& \le (1-\rho)\rho\Delta_{u/\delta_2, u/\delta_1}^{{H_F}}\mathcal{T}_1
 + \frac{2\rho\epsilon_m}{1-\rho(1-\epsilon_m)} - (1-\rho)\mathcal{T}_2.
\end{align*}
where $\Delta_{a,b} := \sup_{0\le s\le a}\left\{H_F(s + (b-a)) - H_F(s)\right\}$, $\epsilon_m(\delta)=\int_0^\delta G_m(s)\dd s$ and
\begin{align*}
 \mathcal{T}_1&:=\frac{1-\delta_1}{1-\rho H_F(u/\delta_1)(1-\delta_1)} - \frac{1-\delta_1 - \epsilon_m + \epsilon_m(\delta_1)}{1-\rho H_F(u/\delta_1)(1-\delta_1 - \epsilon_m + \epsilon_m(\delta_1))}\\
 \mathcal{T}_2&:=\frac{1}{1-(1-\delta_1)\rho}-\frac{2}{1-(1-\delta_1-\epsilon_m+\epsilon_m(\delta_1))\rho} 
  + \frac{1}{\left(\widehat G_m(\delta_2)-1+\epsilon_m\right)\rho}.
\end{align*}
\end{theorem}
The construction of this particular bound requires the selection of two values $\delta_1$ and
$\delta_2$ provided in Corollary \ref{cor:epsilon12}.
In general, it will not be possible to write down a closed-form expression for such values but
in practice this can be easily determined numerically.
Recall that an explicit expression for the term $\epsilon_m(\delta)$ can be found in Lemma
\ref{epsilon.delta}.
\begin{proof}
Recall that
\begin{align*}
& \left| \widehat F^{*n}(u) - \widehat{F\star G}^{*n}(u)\right|
=\left| (H_F\star U)^{*n}(u) - (H_F\star \widehat G_m)^{*n}(u)\right|\nonumber\\
&\qquad \le \idotsint\limits_{\RL^n} \mathbb{P}(s_1 X_1' + \dots + s_n X_n'\le u ) \left| \prod_{i=1}^n {\Ind}_{[0,1)}(s_i)-\prod_{i=1}^n \widehat{g}_m(s_i)    \right| \dd s_1 \dots \dd s_n.
\end{align*}
Split the last integral in two parts: over $[\delta_1, \delta_2]^n$ and over $[0,\infty)^n\setminus [\delta_1, \delta_2]^n$,  bound the  
first one using Lemma \ref{Lem:areasmallFs} and the second one using  Lemma \ref{lem:delta12} and Corollary \ref{Cor.Int}.
Then apply Pollaczeck--Khinchine formula and sum the geometric series.
\end{proof}

\subsection{Bound for $|H_F-H_\Pi|$}

As stated in subsection \ref{sec: Second error}, the result of Theorem \ref{second error bound} depends on the availability of $|H_F-H_\Pi|$. In the following we state a bound for such a quantity in the case 
where an explicit expression for $|H_F-H_\Pi|$ is not available or too difficult to compute.

\begin{Lemma}\label{Lemma.Dis.Moment}
 Let $\Pi$ be defined as in Definition \ref{Def: step distribution} and
 define  $\Delta_k H_F:=H_F(s_k)-H_F(s_{k-1})$.  Then
\begin{align*}
  \sup_{u\le s<\infty} \left|\overline H_F(s)-\overline H_\Pi(s)\right|
    & \le \sup_{K\le k<\infty}\Delta_k H_F+
      \frac{|\mu_\Pi-\mu_F|\cdot\Exp[S;S> u]}{\mu_\Pi\mu_F}+\frac{\left|\Exp[X;X> s_{K}]-\Exp[S;S> s_{K}]\right|}{\mu_F} ,\\
   \sup_{0<s\le u}\left|H_F(s)-H_\Pi(s)\right|   
    &\le\sup_{0\le k\le K}\Delta_k H_F
    +\frac{|\mu_\Pi-\mu_F|\cdot\Exp[S;S\le u]}{\mu_\Pi\mu_F}+
      \frac{|\Exp[S;S\le u]-\Exp[X;X\le u]|}{\mu_F}.  
   \end{align*}
   Moreover, if $\mu_F=\mu_\Pi$, then
\begin{align*}
   \sup_{u\le s<\infty} \left|\overline H_F(s)-\overline H_\Pi(s)\right| 
    &\le \sup_{K\le k<\infty}\Delta_k H_F+
      \frac{\left|\Exp[X;X> s_{K}]-\Exp[S;S> s_{K}]\right|}{\mu_F}, \\
   \sup_{0<s\le u}\left|H_F(s)-H_\Pi(s)\right|    
    &\le\sup_{0\le k\le K}\Delta_k H_F+
      \frac{|\Exp[S;S\le u]-\Exp[X;X\le u]|}{\mu_F}.  
   \end{align*}
\end{Lemma}

Notice that the particular selection of $\Pi$ implies that it is possible to select partitions for which
$\mu_\Pi=\mu_F$.  Also, recall that when $\xi(m)\to\infty$, then $\epsilon_m\to0$, so for $\xi(m)$ sufficiently large, the bound decreases 
as $\left|\Exp[X;X> s_{K}]-\Exp[S;S> s_{K}]\right|$ becomes smaller.  The last is achieved if the tail probability of $H_\Pi$ gets 
\emph{closer} to the tail probability of $H_F$.
\begin{proof}
Since $K\in\nat$ is such that $s_K=u$ then
 \begin{align}
  \left|H_{F}\star \widehat G_m(u) - H_{\Pi}\star \widehat G_m(u)\right|&=
   \left|\int_0^\infty H_F(u/s)d\widehat G(s)-\int_0^\infty H_\Pi(u/s)d\widehat G(s)\right|\nonumber\\
   &\le \int_0^\infty \left|H_F(u/s)-H_\Pi(u/s)\right|d\widehat G(s)\nonumber\\
   &\le \sup_{u\le s<\infty} \left|H_F(s)-H_\Pi(s)\right|\int_0^{1}d\widehat G(s)\nonumber\\ 
   &\qquad\qquad  +\sup_{0<s\le u}\left|H_F(s)-H_\Pi(s)\right|\int_{1}^\infty  \dd\widehat G(s)\nonumber\\
   &= \sup_{u\le s<\infty} \left|\overline H_\Pi(s)-\overline H_F(s)\right|\int_0^{1}d\widehat G(s)\label{eq.IntTail.1}\\
   &\qquad\qquad  +\sup_{0<s\le u}\left|H_F(s)-H_\Pi(s)\right|\int_{1}^\infty d\widehat G(s).\label{eq.IntTail.2}
 \end{align}
Observe that for all $0< s<\infty$ there exist $k$ such
that $t_{k}\le s< t_{k+1}$, so
 \begin{align*}
   \left|\overline H_\Pi(s)-\overline H_F(s)\right|
     &\le  \max\{   \left|\overline H_F(s_{k})-\overline H_\Pi(s_k)\right|,   \left|\overline H_F(s_{k+1})-\overline H_\Pi(s_k)\right|\}. 
 \end{align*}
Using the previous identity we first constructing a bound for \eqref{eq.IntTail.1}. 
 \begin{align*}
   |\overline H_\Pi(s_{k})-\overline H_F(s_k)|
     &\le    |\overline H_\Pi(s_{k})-\overline H_F(s_{k+1})|+|\overline H_F(s_{k})-\overline H_F(s_{k+1})|\\
     &\le    |\overline H_\Pi(s_{k})-\overline H_F(s_{k+1})|+\Delta_{k} H_F,
 \end{align*}
 where $\Delta_{k} H_F:=H_F(s_{k+1})-H_F(s_{k})$, and
 in consequence
 \begin{equation*}
  \sup_{u\le s<\infty}\left|\overline H_\Pi(s)-\overline H_F(s)\right|\le
   \sup_{ K\le k<\infty}|\overline H_\Pi(s_{k})-\overline H_F(s_{k+1})|+\sup_{ K\le k<\infty}\Delta_k H_F.
 \end{equation*} 
Next observe that
 \begin{align*}
   \sup_{K\le k< \infty}|\overline H_\Pi(s_k)-\overline H_F(s_{k+1})|
     &=  \sup_{K\le k< \infty}\left|\sum_{i=k+1}^{\infty} \frac{s_{i}\pi_i}{\mu_\Pi}-\int_{s_{k+1}}^\infty\frac{t dF(t)}{\mu_F}\right|\\       
     &=  \sup_{K\le k< \infty}\left|\sum_{i=k+1}^{\infty}\int_{s_i}^{s_{i+1}}  \left(\frac{s_{i}}{\mu_\Pi}-\frac{t}{\mu_F}\right)dF(t)\right|\\
     &\le \frac{1}{\mu_\Pi\mu_F} \sum_{i=K+1}^{\infty}\int_{s_{i}}^{s_{i+1}} |s_{i}\mu_F-t \mu_\Pi|dF(t)\\     
     &\le \frac{1}{\mu_\Pi\mu_F} \sum_{i=K+1}^{\infty}\int_{s_{i}}^{s_{i+1}} \left(|s_{i}\mu_F-s_{i}\mu_\Pi| + |s_{i}\mu_\Pi-t \mu_\Pi|\right)dF(t)\\
     &\le \frac{|\mu_F-\mu_\Pi|}{\mu_\Pi\mu_F} \sum_{i=K+1}^{\infty}s_{i}\int_{s_{i}}^{s_{i+1}}dF(t)
           +\frac{1}{\mu_F} \sum_{i=K+1}^{\infty}\int_{s_{i}}^{s_{i+1}} |s_{i}-t|dF(t) \\
     &\le \frac{|\mu_\Pi-\mu_F|\Exp[S;S> s_{K}]}{\mu_\Pi\mu_F}+\frac{\left|\Exp[X;X> s_{K}]-\Exp[S;S> s_{K}]\right|}{\mu_F}.
 \end{align*}

Therefore, 
\begin{align*}
  &\sup_{u\le s<\infty} \left|\overline H_\Pi(s)-\overline H_F(s)\right|\\
  &\qquad\le \sup_{K\le k<\infty}\Delta_k H_F+
      \frac{|\mu_\Pi-\mu_F|\cdot\Exp[S;S> u]}{\mu_\Pi\mu_F}+\frac{\left|\Exp[X;X> s_{K}]-\Exp[S;S> s_{K}]\right|}{\mu_F}.
\end{align*}

Our construction for the bound for \eqref{eq.IntTail.2} is analogous.  Note that
 \begin{align*}
   |H_F(s_{k+1})-H_\Pi(s_k)|
     &\le    |H_F(s_{k+1})-H_F(s_k)|+|H_F(s_{k})-H_\Pi(s_k)|\\
     &\le    \Delta_{k} H_F+|H_F(s_{k})-H_\Pi(s_k)|,
 \end{align*}
so 
\begin{equation*}
  \sup_{0<s\le u}\left|H_F(s)-H_\Pi(s)\right|\le
   \sup_{0\le k\le K}\Delta_k H_F+\sup_{0\le k\le K}|H_F(s_{k})-H_\Pi(s_k)|,
 \end{equation*} 
where $s_0=\inf\{s:F(s)>0\}$.  Next observe that
 \begin{align*}
   \sup_{0\le k\le K}|H_F(s_{k})-H_\Pi(s_k)|
     &=  \sup_{0\le k\le K}\left|\int_0^{s_{k}}\frac{t dF(t)}{\mu_F}-\sum_{i=1}^{k} \frac{s_{i}\pi_i}{\mu_\Pi}\right|\\       
     &=  \sup_{0\le k\le K}\left|\sum_{i=1}^{k}\int_{s_i-1}^{s_{i}}  \left(\frac{t }{\mu_F}-\frac{s_{i}}{\mu_\Pi}\right)dF(t)\right|\\
     &\le \frac{1}{\mu_\Pi\mu_F} \sum_{i=1}^{K}\int_{s_{i-1}}^{s_{i}} |t\mu_\Pi-s_{i}\mu_F|dF(t)\\     
     &\le \frac{1}{\mu_\Pi\mu_F} \sum_{i=1}^{K}\int_{s_{i-1}}^{s_{i}} |t\mu_\Pi-s_{i}\mu_\Pi|+\left(|s_{i}\mu_\Pi-s_{i}\mu_F|\right)dF(t)\\
     &\le \frac{1}{\mu_F} \sum_{i=1}^{K}\int_{s_{i-1}}^{s_{i}} |s_{i}-t|dF(t) 
        +\frac{|\mu_\Pi-\mu_F|}{\mu_\Pi\mu_F} \sum_{i=1}^{K}s_{i}\int_{s_{i-1}}^{s_{i}}dF(t)\\
     &\le \frac{|\Exp[S;S\le s_K]-\Exp[X;X\le s_K]|}{\mu_F}+\frac{|\mu_\Pi-\mu_F|\cdot\Exp[S;S\le s_K]}{\mu_\Pi\mu_F}.
 \end{align*}

Therefore, 
\begin{align*}
  &\sup_{0\le s\le u/\delta} \left|H_F(s)-H_\Pi(s)\right|\\
  &\qquad\le \sup_{0\le k\le K}\Delta_k H_F+
      \frac{|\Exp[S;S\le u]-\Exp[X;X\le u]|}{\mu_F}+\frac{|\mu_\Pi-\mu_F|\cdot\Exp[S;S\le u]}{\mu_\Pi\mu_F}.
\end{align*}

\end{proof}

\end{document}